\documentclass[10pt,a4paper]{article}
\usepackage[utf8]{inputenc}
\usepackage[english]{babel}
\usepackage[T1]{fontenc}

\usepackage{amsmath}
\usepackage{amsthm}
\usepackage{amsfonts}
\usepackage{amssymb}
\usepackage{scrextend}
\usepackage[colorlinks=true, allcolors=blue]{hyperref}
\usepackage{cleveref}
\usepackage{graphicx}
\usepackage{enumerate}
\usepackage{lmodern}
\usepackage{wrapfig}
\usepackage{mathrsfs}
\usepackage{subcaption}
\usepackage[onelanguage,ruled,vlined,linesnumbered]{algorithm2e}
\usepackage{siunitx}
\usepackage{pdfpages}
\usepackage{float}
\usepackage[colorinlistoftodos]{todonotes}
\usepackage{comment}
\usepackage[parfill]{parskip}
\usepackage{bm}
\usepackage{xcolor}
\usepackage{makecell}
\usepackage{tablefootnote}

\usepackage{diagbox}

\usepackage{tikz}
\usetikzlibrary{shapes,graphs,graphs.standard}

\newcommand{\db}[2]{#1\!\leftrightarrow\!#2}
\newcommand{\ru}[1]{\textbf{R#1}}

\newtheorem{theorem}{Theorem}
\newtheorem{corollary}[theorem]{Corollary}
\newtheorem{definition}[theorem]{Definition}

\newtheorem{proposition}[theorem]{Proposition}
\crefname{proposition}{Proposition}{Propositions}
\newtheorem{lemma}[theorem]{Lemma}
\newtheorem{conjecture}[theorem]{Conjecture}

\newtheorem{observation}[theorem]{Observation}
\newtheorem{claim}[theorem]{Claim}
\crefname{claim}{claim}{claims}
\Crefname{claim}{Claim}{Claims}

\usepackage[left=2cm,right=2cm,top=2cm,bottom=2cm]{geometry}

\renewcommand{\thesubfigure}{\roman{subfigure}}
\DeclareMathOperator{\ad}{ad}
\DeclareMathOperator{\mad}{mad}

\usepackage{authblk}

\title{$2$-distance $(\Delta+1)$-coloring of sparse graphs using the potential method}

\author[1]{Hoang La\thanks{xuan-hoang.la@lirmm.fr}}
\author[1]{Mickael Montassier\thanks{mickael.montassier@lirmm.fr}}
\affil[1]{LIRMM, Université de Montpellier, CNRS, Montpellier, France}
\begin{document}
  \maketitle

\begin{abstract}
A $2$-distance $k$-coloring of a graph is a proper $k$-coloring of the vertices where vertices at distance at most 2 cannot share the same color. We prove the existence of a $2$-distance ($\Delta+1$)-coloring for graphs with maximum average degree less than $\frac{18}{7}$ and maximum degree $\Delta\geq 7$. As a corollary, every planar graph with girth at least $9$ and $\Delta\geq 7$ admits a $2$-distance $(\Delta+1)$-coloring. The proof uses the potential method to reduce new configurations compared to classic approaches on $2$-distance coloring.
\end{abstract}

\section{Introduction}

A \emph{$k$-coloring} of the vertices of a graph $G=(V,E)$ is a map $\phi:V \rightarrow\{1,2,\dots,k\}$. A $k$-coloring $\phi$ is a \emph{proper coloring}, if and only if, for all edge $xy\in E,\phi(x)\neq\phi(y)$. In other words, no two adjacent vertices share the same color. The \emph{chromatic number} of $G$, denoted by $\chi(G)$, is the smallest integer $k$ such that $G$ has a proper $k$-coloring.  A generalization of $k$-coloring is $k$-list-coloring.
A graph $G$ is {\em $L$-list colorable} if for a
given list assignment $L=\{L(v): v\in V(G)\}$ there is a proper
coloring $\phi$ of $G$ such that for all $v \in V(G), \phi(v)\in
L(v)$. If $G$ is $L$-list colorable for every list assignment $L$ with $|L(v)|\ge k$ for all $v\in V(G)$, then $G$ is said to be {\em $k$-choosable} or \emph{$k$-list-colorable}. The \emph{list chromatic number} of a graph $G$ is the smallest integer $k$ such that $G$ is $k$-choosable. List coloring can be very different from usual coloring as there exist graphs with a small chromatic number and an arbitrarily large list chromatic number.

In 1969, Kramer and Kramer introduced the notion of 2-distance coloring \cite{kramer2,kramer1}. This notion generalizes the ``proper'' constraint (that does not allow two adjacent vertices to have the same color) in the following way: a \emph{$2$-distance $k$-coloring} is such that no pair of vertices at distance at most 2 have the same color (similarly to proper $k$-list-coloring, one can also define \emph{$2$-distance $k$-list-coloring}). The \emph{$2$-distance chromatic number} of $G$, denoted by $\chi^2(G)$, is the smallest integer $k$ so that $G$ has a 2-distance $k$-coloring.

For all $v\in V$, we denote $d_G(v)$ the degree of $v$ in $G$ and by $\Delta(G) = \max_{v\in V}d_G(v)$ the maximum degree of a graph $G$. For brevity, when it is clear from the context, we will use $\Delta$ (resp. $d(v)$) instead of $\Delta(G)$ (resp. $d_G(v)$). 
One can observe that, for any graph $G$, $\Delta+1\leq\chi^2(G)\leq \Delta^2+1$. The lower bound is trivial since, in a 2-distance coloring, every neighbor of a vertex $v$ with degree $\Delta$, and $v$ itself must have a different color. As for the upper bound, a greedy algorithm shows that $\chi^2(G)\leq \Delta^2+1$. Moreover, this bound is tight for some graphs, for example, Moore graphs of type $(\Delta,2)$, which are graphs where all vertices have degree $\Delta$, are at distance at most two from each other, and the total number of vertices is $\Delta^2+1$. See \Cref{tight upper bound figure}.

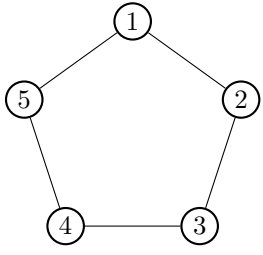
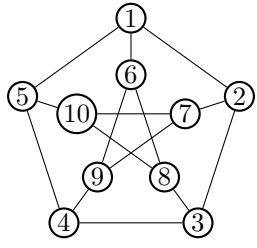
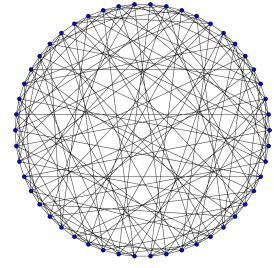
\begin{figure}[htbp]
\begin{center}
\begin{subfigure}[t]{5cm}
\centering
\begin{tikzpicture}[every node/.style={circle,thick,draw,minimum size=1pt,inner sep=2}]
  \graph[clockwise, radius=1.5cm] {subgraph C_n [n=5,name=A] };
\end{tikzpicture}
\caption{The Moore graph of type (2,2):\\ the odd cycle $C_5$}
\end{subfigure}
\qquad
\begin{subfigure}[t]{5cm}
\centering
\begin{tikzpicture}[every node/.style={circle,thick,draw,minimum size=1pt,inner sep=1}]
  \graph[clockwise, radius=1.5cm] {subgraph C_n [n=5,name=A] };
  \graph[clockwise, radius=0.75cm,n=5,name=B] {1/"6", 2/"7", 3/"8", 4/"9", 5/"10" };

  \foreach \i [evaluate={\j=int(mod(\i+6,5)+1)}]
     in {1,2,3,4,5}{
    \draw (A \i) -- (B \i);
    \draw (B \j) -- (B \i);
  }
\end{tikzpicture}
\caption{The Moore graph of type (3,2):\\ the Petersen graph.}
\end{subfigure}
\qquad
\begin{subfigure}[t]{5cm}
\centering
\includegraphics[scale=0.12]{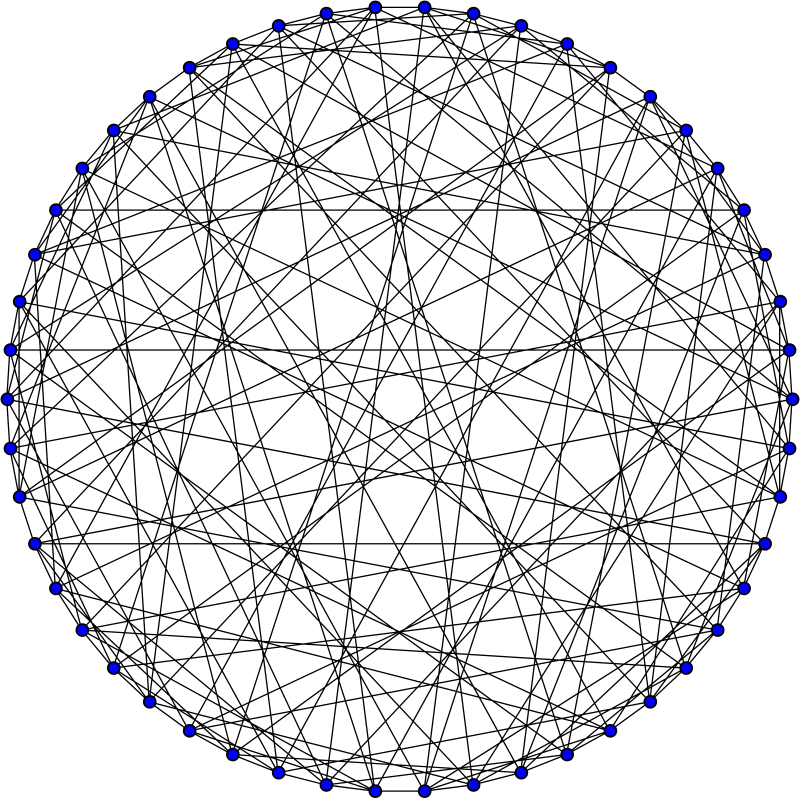}
\caption{The Moore graph of type (7,2):\\ the Hoffman-Singleton graph.}
\end{subfigure}
\caption{Examples of Moore graphs for which $\chi^2=\Delta^2+1$.}
\label{tight upper bound figure}
\end{center}
\end{figure}

By nature, $2$-distance colorings and the $2$-distance chromatic number of a graph depend a lot on the number of vertices in the neighborhood of every vertex. More precisely, the ``sparser'' a graph is, the lower its $2$-distance chromatic number will be. One way to quantify the sparsity of a graph is through its maximum average degree. The \emph{average degree} $\ad$ of a graph $G=(V,E)$ is defined by $\ad(G)=\frac{2|E|}{|V|}$. The \emph{maximum average degree} $\mad(G)$ is the maximum, over all subgraphs $H$ of $G$, of $\ad(H)$. Another way to measure the sparsity is through the girth, i.e. the length of a shortest cycle. We denote $g(G)$ the girth of $G$. Intuitively, the higher the girth of a graph is, the sparser it gets. These two measures can actually be linked directly in the case of planar graphs.

A graph is \emph{planar} if one can draw its vertices with points on the plane, and edges with curves intersecting only at its endpoints. When $G$ is a planar graph, Wegner conjectured in 1977 that  $\chi^2(G)$ becomes linear in $\Delta(G)$:

\begin{conjecture}[Wegner \cite{wegner}]
\label{conj:Wegner}
Let $G$ be a planar graph with maximum degree $\Delta$. Then,
$$
\chi^2(G) \leq \left\{
    \begin{array}{ll}
        7, & \mbox{if } \Delta\leq 3, \\
        \Delta + 5, & \mbox{if } 4\leq \Delta\leq 7,\\
        \left\lfloor\frac{3\Delta}{2}\right\rfloor + 1, & \mbox{if } \Delta\geq 8.
    \end{array}
\right.
$$
\end{conjecture}

The upper bound for the case where $\Delta\geq 8$ is tight (see \Cref{wegner figure}(i)). Recently, the case $\Delta\leq 3$ was proved by Thomassen \cite{tho18}, and by Hartke \textit{et al.} \cite{har16} independently. For $\Delta\geq 8$, Havet \textit{et al.} \cite{havet} proved that the bound is $\frac{3}{2}\Delta(1+o(1))$, where $o(1)$ is as $\Delta\rightarrow\infty$ (this bound holds for 2-distance list-colorings). \Cref{conj:Wegner} is known to be true for some subfamilies of planar graphs, for example $K_4$-minor free graphs \cite{lwz03}.

\begin{figure}[htbp]
\begin{subfigure}[b]{0.48\textwidth}
\centering
\begin{tikzpicture}[scale=0.4]
\begin{scope}[every node/.style={circle,thick,draw,minimum size=1pt,inner sep=2}]
    \node[fill] (y) at (0,0) {};
    \node[fill] (z) at (5,0) {};
    \node[fill] (x) at (2.5,4.33) {};

    \node[fill] (xy) at (1.25,2.165) {};
    \node[fill] (yz) at (2.5,0) {};
    \node[fill] (zx) at (3.75,2.165) {};

    \node[fill,label={[label distance=-1cm]above:$\lfloor\frac{\Delta}{2}\rfloor-1$ vertices}] (xy1) at (-2.5,4.33) {};
    \node[fill] (xy2) at (-1.5625,3.78875) {};
    \node[fill] (xy3) at (-0.625,3.2475) {};

    \node[fill,label={[label distance=-0.7cm]above:$\lceil\frac{\Delta}{2}\rceil$ vertices}] (zx1) at (7.5,4.33) {};
    \node[fill] (zx2) at (6.5625,3.78875) {};
    \node[fill] (zx3) at (5.625,3.2475) {};

    \node[fill,label={[label distance = -0.7cm]below:$\lfloor\frac{\Delta}{2}\rfloor$ vertices}] (yz1) at (2.5,-4.33) {};
    \node[fill] (yz2) at (2.5,-3.2475) {};
    \node[fill] (yz3) at (2.5,-2.165) {};

\end{scope}

\begin{scope}[every edge/.style={draw=black,thick}]
    \path (x) edge (y);
    \path (y) edge (z);
    \path (z) edge (x);

    \path (x) edge[bend left] (y);

    \path (x) edge (xy1);
    \path (x) edge (xy2);
    \path (x) edge (xy3);

    \path (y) edge (xy1);
    \path (y) edge (xy2);
    \path (y) edge (xy3);

    \path (x) edge (zx1);
    \path (x) edge (zx2);
    \path (x) edge (zx3);

    \path (z) edge (zx1);
    \path (z) edge (zx2);
    \path (z) edge (zx3);

    \path (y) edge (yz1);
    \path (y) edge (yz2);
    \path (y) edge (yz3);

    \path (z) edge (yz1);
    \path (z) edge (yz2);
    \path (z) edge (yz3);

    \path[dashed] (xy) edge (xy3);
    \path[dashed] (yz) edge (yz3);
    \path[dashed] (zx) edge (zx3);
\end{scope}
\draw[rotate=-30] (-0.625-1.4,3.2475-0.7) ellipse (3cm and 0.5cm);
\draw[rotate=30] (5+0.625+1,3.2475-3.25) ellipse (3cm and 0.5cm);
\draw[rotate=90] (-2,-2.5) ellipse (3cm and 0.5cm);
\end{tikzpicture}
\vspace{-0.9cm}
\caption{A graph with girth 3 and $\chi^2=\lfloor\frac{3\Delta}{2}\rfloor+1$}
\end{subfigure}
\begin{subfigure}[b]{0.48\textwidth}
\centering
\begin{tikzpicture}[scale=0.4]
\begin{scope}[every node/.style={circle,thick,draw,minimum size=1pt,inner sep=2}]
    \node[fill] (y) at (0,0) {};
    \node[fill] (z) at (5,0) {};
    \node[fill] (x) at (2.5,4.33) {};

    \node[fill] (xy) at (1.25,2.165) {};
    \node[fill] (yz) at (2.5,0) {};
    \node[fill] (zx) at (3.75,2.165) {};

    \node[fill,label={[label distance=-1cm]above:$\lfloor\frac{\Delta}{2}\rfloor-1$ vertices}] (xy1) at (-2.5,4.33) {};
    \node[fill] (xy2) at (-1.5625,3.78875) {};
    \node[fill] (xy3) at (-0.625,3.2475) {};

    \node[fill,label={[label distance=-0.7cm]above:$\lceil\frac{\Delta}{2}\rceil$ vertices}] (zx1) at (7.5,4.33) {};
    \node[fill] (zx2) at (6.5625,3.78875) {};
    \node[fill] (zx3) at (5.625,3.2475) {};

    \node[fill,label={[label distance = -0.7cm]below:$\lfloor\frac{\Delta}{2}\rfloor$ vertices}] (yz1) at (2.5,-4.33) {};
    \node[fill] (yz2) at (2.5,-3.2475) {};
    \node[fill] (yz3) at (2.5,-2.165) {};

\end{scope}

\begin{scope}[every edge/.style={draw=black,thick}]
    \path (x) edge (y);
    \path (y) edge (z);
    \path (z) edge (x);

    \path (x) edge (xy1);
    \path (x) edge (xy2);
    \path (x) edge (xy3);

    \path (y) edge (xy1);
    \path (y) edge (xy2);
    \path (y) edge (xy3);

    \path (x) edge (zx1);
    \path (x) edge (zx2);
    \path (x) edge (zx3);

    \path (z) edge (zx1);
    \path (z) edge (zx2);
    \path (z) edge (zx3);

    \path (y) edge (yz1);
    \path (y) edge (yz2);
    \path (y) edge (yz3);

    \path (z) edge (yz1);
    \path (z) edge (yz2);
    \path (z) edge (yz3);

    \path[dashed] (xy) edge (xy3);
    \path[dashed] (yz) edge (yz3);
    \path[dashed] (zx) edge (zx3);
\end{scope}
\draw[rotate=-30] (-0.625-1.4,3.2475-0.7) ellipse (3cm and 0.5cm);
\draw[rotate=30] (5+0.625+1,3.2475-3.25) ellipse (3cm and 0.5cm);
\draw[rotate=90] (-2,-2.5) ellipse (3cm and 0.5cm);
\end{tikzpicture}
\vspace{-0.9cm}
\caption{A graph with girth 4 and $\chi^2=\lfloor\frac{3\Delta}{2}\rfloor-1$.}
\end{subfigure}
\caption{Graphs with $\chi^2\approx \frac32 \Delta$}
\label{wegner figure}
\end{figure}
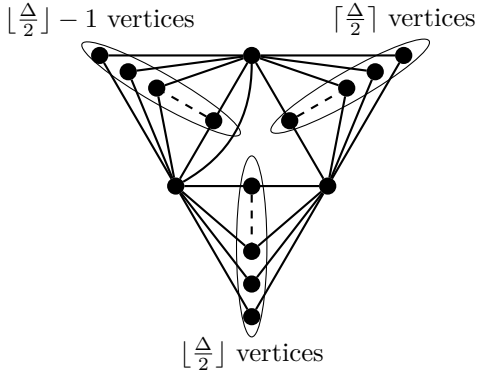
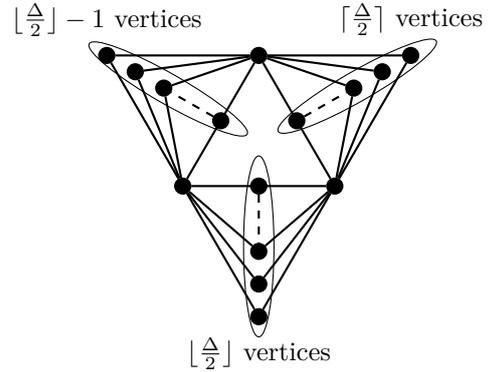

Wegner's conjecture motivated extensive researches on $2$-distance chromatic number of sparse graphs, either of planar graphs with high girth or of graphs with upper bounded maximum average degree which are directly linked due to \Cref{maximum average degree and girth proposition}.

\begin{proposition}[Folklore]\label{maximum average degree and girth proposition}
For every planar graph $G$, $(\mad(G)-2)(g(G)-2)<4$.
\end{proposition}

As a consequence, any theorem with an upper bound on $\mad(G)$ can be translated to a theorem with a lower bound on $g(G)$ under the condition that $G$ is planar. Many results have taken the following form: \textit{every graph $G$ of $\mad(G)\leq m_0$ and $\Delta(G)\geq \Delta_0$ satisfies $\chi^2(G)\leq \Delta(G)+c(m_0,\Delta_0)$ where $c(m_0,\Delta_0)$ is a small constant depending only on $m_0$ and $\Delta_0$}. Due to \Cref{maximum average degree and girth proposition}, as a corollary, we have the same results on planar graphs of girth $g\geq g_0(m_0)$ where $g_0$ depends on $m_0$. \Cref{recap table 2-distance} shows all known such results, up to our knowledge, on the $2$-distance chromatic number of planar graphs with fixed girth, either proven directly for planar graphs with high girth or came as a corollary of a result on graphs with bounded maximum average degree.

\begin{table}[H]
\begin{center}
\scalebox{0.7}{%
\begin{tabular}{||c||c|c|c|c|c|c|c|c||}
\hline
\backslashbox{$g_0$ \kern-1em}{\kern-1em $\chi^2(G)$} & $\Delta+1$ & $\Delta+2$ & $\Delta+3$ & $\Delta+4$ & $\Delta+5$ & $\Delta+6$ & $\Delta+7$ & $\Delta+8$\\
\hline \hline
$3$ & \slashbox{\phantom{\ \ \ \ \ }}{} & & &$\Delta=3$ \cite{tho18,har16}& & & & \\
\hline
$4$ & \slashbox{\phantom{\ \ \ \ \ }}{} & & & & & & & \\
\hline
$5$ & \slashbox{\phantom{\ \ \ \ \ }}{} &$\Delta\geq 10^7$ \cite{bon19}\footref{list footnote} &$\Delta\geq 339$ \cite{don17b} &$\Delta\geq 312$ \cite{don17} &$\Delta\geq 15$ \cite{bu18b}\tablefootnote{\label{other footnote}Corollaries of more general colorings of planar graphs.} &$\Delta\geq 12$ \cite{bu16}\footref{list footnote} & $\Delta\neq 7,8$ \cite{don17} &all $\Delta$ \cite{dl16}\\
\hline
$6$ & \slashbox{\phantom{\ \ \ \ \ }}{} &$\Delta\geq 17$ \cite{bon14}\footref{mad footnote} &$\Delta\geq 9$ \cite{bu16}\footref{list footnote} & &all $\Delta$ \cite{bu11} & & & \\
\hline
$7$ & $\Delta\geq 16$ \cite{iva11}\tablefootnote{\label{list footnote}Corollaries of 2-distance list-colorings of planar graphs.} & & &$\Delta=4$ \cite{cra13}\tablefootnote{\label{mad list footnote}Corollaries of 2-distance list-colorings of graphs with a bounded maximum average degree.} & & & & \\
\hline
$8$ & $\Delta\geq 9$\cite{lmpv19}\footref{other footnote}& &$\Delta=5$ \cite{bu15}\footref{mad list footnote} & & & & & \\
\hline
$9$ &\makecell{$\Delta\geq 8$ \cite{bon13}\footref{mad footnote}\\\textcolor{red}{$\Delta\geq 7$}\tablefootnote{Corollary of our result.}} &$\Delta=5$ \cite{bu15}\footref{mad list footnote} &$\Delta=3$ \cite{cra07}\footref{list footnote} & & & & & \\
\hline
$10$ & $\Delta\geq 6$ \cite{iva11}\footref{list footnote} & & & & & & & \\
\hline
 $11$ & &$\Delta=4$ \cite{cra13}\footref{mad list footnote} & & & & & & \\
\hline
$12$ & $\Delta=5$ \cite{iva11}\footref{list footnote} &$\Delta=3$ \cite{bi12}\footref{list footnote} & & & & & &\\
\hline
$13$ & & & & & & & &\\
\hline
$14$ &$\Delta\geq 4$ \cite{bon13}\tablefootnote{\label{mad footnote}Corollaries of 2-distance colorings of graphs with a bounded maximum average degree.} & & & & & & &\\
\hline
$\dots$ & & & & & & & & \\
\hline
$22$ & $\Delta=3$ \cite{bi12bis}\footref{list footnote} & & & & & & & \\
\hline
\end{tabular}}
\caption{The latest results with a coefficient 1 before $\Delta$ in the upper bound of $\chi^2$.}
\label{recap table 2-distance}
\end{center}
\end{table} 

For example, the result from line ``7'' and column ``$\Delta + 1$'' from \Cref{recap table 2-distance} reads as follows : ``\emph{every planar graph $G$ of girth at least 7  and of $\Delta$ at least 16 satisfies $\chi^2(G)\leq \Delta+1$}''. The crossed out cases in the first column correspond to the fact that, for $g_0\leq 6$, there are planar graphs $G$ with $\chi^2(G)=\Delta+2$ for arbitrarily large $\Delta$ \cite{bor04,dvo08b}. The lack of results for $g = 4$ is due to the fact that the graph in \Cref{wegner figure}(ii) has girth 4, and $\chi^2=\lfloor\frac{3\Delta}{2}\rfloor-1$ for all $\Delta$.

We are interested in the case $\chi^2(G)=\Delta+1$ as $\Delta+1$ is a trivial lower bound for $\chi^2(G)$. In particular, we were looking for the smallest integer $\Delta_0$ such that every graph with maximum degree $\Delta\geq \Delta_0$ and $\mad\leq \frac{18}{7}$ (which contains planar graphs with $\Delta\geq\Delta_0$ and girth at least $9$) can be $2$-distance colored with $\Delta+1$ colors. Borodin \textit{et al.}~\cite{bor08} showed that planar graphs of girth at least $9$ and $\Delta\geq 10$ are $2$-distance $(\Delta+1)$-colorable in 2008. In 2011, Ivanova~\cite{iva11} improved on the result with a simpler proof that planar graphs of girth at least $8$ and $\Delta\geq 10$ are $2$-distance $(\Delta+1)$-colorable. Later on, in 2014, Bonamy \textit{et al.}~\cite{bon13} improved on this result once again by proving that graphs with $\mad<\frac{18}{7}$ and $\Delta\geq 8$ are $2$-distance $\Delta+1$-colorable. In this paper, we will improve this result to graphs with $\mad<\frac{18}{7}$ and $\Delta\geq 7$ in \Cref{main theorem}. But most importantly, that breakthrough is obtained by using a new approach based on the potential method.

All of these results and most of the results in \Cref{recap table 2-distance} are proven using the discharging method. Due to the extensive amount of work done on this subject, the classic discharging method is reaching its limit. The \emph{discharging method} assigns a certain charge to each object (often vertices and faces when the graph is planar) of a minimal counter-example $G$ to the result we want to prove. Then, using either Euler's formula or the upper bound on the maximum average degree, we can prove that the total amount of charges is negative. However, by redistributing these charges via \emph{discharging rules} that do not modify the total sum, we can prove that we have a nonnegative amount of charges under the reducibility of some configurations, which results in a contradiction. Since the initial total amount of charges is fixed, the improvements on these type of results rely on the reduction of new configurations and reducing a configuration relies on extending a precoloring of a subgraph of $G$. Until now, we have always assumed the worst case scenario for the precoloring. However, these assumptions can only get us so far when we can find unextendable precolorations. In order to avoid the worst case scenario, we need to add some vertices and edges to our subgraph but we might run into the risk of increasing our maximum average degree. Then came the \emph{potential method}, which introduces a potential function that can, more precisely, quantify the local maximum average degree in our subgraph, thus allowing us to add edges and vertices while staying in the same class of graphs. This breakthrough allowed for new configurations to become reducible and thus, improving on the limit of what the classic discharging method was able to reach.  

Our main result is the following:
\begin{theorem} \label{main theorem}
If $G$ is a graph with $\mad(G)< \frac{18}{7}$, then $G$ is $2$-distance $(\Delta(G)+1)$-colorable for $\Delta(G)\geq 7$.
\end{theorem}

For planar graphs, we obtain the following corollary:
\begin{corollary} \label{main corollary}
If $G$ is a planar graph with $g(G)\geq 9$, then $G$ is 2-distance $(\Delta(G)+1)$-colorable for $\Delta(G)\geq 7$.
\end{corollary}

Since Bonamy, Lévêque, and Pinlou has already proven in \cite{bon13} that:
\begin{theorem}[Bonamy, Lévêque, Pinlou \cite{bon13}]
If $G$ is a graph with $\mad(G)< \frac{18}{7}$, then $G$ is list $2$-distance $(\Delta(G)+1)$-colorable for $\Delta(G)\geq 8$.
\end{theorem}

We will prove the following, which is a stronger version with $\mad(G)\leq \frac{18}{7}$ instead of $\mad(G)<\frac{18}{7}$:
\begin{theorem} \label{main theorem2}
If $G$ is a graph with $\mad(G)\leq \frac{18}{7}$, then $G$ is $2$-distance $(\Delta(G)+1)$-colorable for $\Delta(G) = 7$.
\end{theorem}

To prove \Cref{main theorem2}, let us define the potential function, which is the key to the potential method.

Let $A\subseteq V(G)$, we define $\rho_G(A) = 9|A| - 7|E(G[A])|$.
Note that $\rho_G(A)\geq 0$ for all $A\subseteq V(G)$ if and only if $\mad(G)\leq \frac{18}{7}$. 
We define \emph{the potential function} $\rho^*_G(A) = \min\{\rho_G(S)|A\subseteq S\subseteq V(G)\}$ for all $A\subseteq V(G)$. Since $\rho_G(A)\geq 0$ for all $A\subseteq V(G)$, the same holds for $\rho^*_G(A)$.

Thus, we will prove the following equivalent version of \Cref{main theorem2}.

\begin{theorem} \label{theorem}
Let $G$ be a graph such that $\rho^*_G(A)\geq 0$ for all $A\subseteq V(G)$, then $G$ is $2$-distance $(\Delta(G)+1)$-colorable for $\Delta(G) = 7$.
\end{theorem}

We study some elementary operations of the potential function in \Cref{sec2}. Then, we will use this potential function coupled with the discharging method as mentionned above to prove that result in \Cref{sec3}.

\section{Elementary operations with the potential function}
\label{sec2}


In this section, we will prove some useful inequalities, that we will use repeatedly in our proof, involving this potential function on a graph $G$ with $\mad(G)\leq \frac{18}{7}$ .


Let $A,S\subseteq V(G)$ such that $A\subseteq S$. Since any $K\subseteq V(G)$ that contains $S$ will also contain $A$, by definition of $\rho^*_G$, we have:
\begin{equation}\label{potsuperset}
\rho^*_G(S) \geq \rho^*_G(A).
\end{equation}


Let $A\subseteq V(G)$ and $H$ be a subgraph of $G$ that contains $A$. Since any subset $S\subseteq V(H)$ that contains $A$ is also a subset of $V(G)$, by definition of $\rho^*_G$, the following ensues:
\begin{equation}\label{potsubgraph}
\rho^*_H(A)\geq \rho^*_G(A).
\end{equation}


Let $A,B\subseteq V(G)$. Since $|A|+|B| = |A\cup B| + |A\cap B|$ and $|E(G[A])|+|E(G[B])| \leq |E(G[A\cup B])| + |E(G[A\cap B])|$, we have $\rho_G(A)+\rho_G(B) \geq \rho_G(A\cup B) +\rho_G(A\cap B)$.

Now, let $A\subseteq S\subseteq V(G)$ and $B\subseteq T\subseteq V(G)$ such that $\rho_G(S) = \rho_G^*(A)$ and $\rho_G(T)=\rho_G^*(B)$. By the previous observation, we have $\rho_G^*(A) + \rho_G^*(B) = \rho_G(S) + \rho_G(T) \geq \rho_G(S\cup T) +\rho_G(S\cap T)$. Since $(A\cup B) \subseteq (S\cup T)$ and $(A \cap B) \subseteq (S\cap T)$, by definition of $\rho_G^*$, we have the following:

\begin{equation}\label{potadd}
\rho_G^*(A)+\rho_G^*(B) \geq \rho_G^*(A\cup B) +\rho_G^*(A\cap B).
\end{equation}


Let $A$ and $S$ be disjoint subsets of $V(G)$ such that $S$ contains (at least) every vertex (not in $A$) that is adjacent to a vertex in $A$. We denote $E(A,S)$ the set of edges between vertices of $A$ and $S$. By definition, $\rho_G(A\cup S) = 9|A\cup S| - 7|E(G[A\cup S])| = 9(|A| + |S|) - 7(|E(G[A])|+|E(G[S])|+|E(A,S)|) = (9|S| - 7|E(G[S])|) + (9|A| - 7|E(G[A])|) - 7|E(A,S)| = \rho_G(S) + \rho_G(A) - 7|E(A,S)|$. Since $\mad(G)\leq \frac{18}{7}$, we know that $\rho_G(A\cup S)\geq 0$. Thus, $\rho_G(S)\geq 7|E(A,S)| - \rho_G(A)$. Observe that $S\subseteq V(G-A)$ and the previous inequality holds for any $K\subseteq V(G-A)$ that contains $S$. Moreover for every $K$ that contains $S$ we have $|E(A,S)|=|E(A,K)|$ by definition of $S$. Hence, the following also holds:


\begin{equation}\label{potlowerbound}
\rho_{G-A}^*(S)\geq 7|E(A,S)| - \rho_G(A).
\end{equation}

\begin{lemma} \label{potlemma}
Suppose graph $H$ verifies $\mad(H)\leq \frac{18}{7}$. Let $k\geq 0$ and $u,v\in V(H)$. Moreover assume that $\rho_H^*(\{u,v\})\geq 7 - 2k$. Let $H'=H+P$ be the graph obtained from $H$ in which we add a path $P$ with $k$ internal vertices of degree 2 between $u$ and $v$ ($P$ is just an edge when $k=0$), then $\mad(H')\leq \frac{18}{7}$ (equivalently, $\forall T \subseteq V(H'), \rho_{H'}(T)\ge 0$). 
\end{lemma}


\begin{proof}

Indeed, every subset $S\subseteq V(H')$ that does not contain any vertex or edges of $P$ is a subset of $V(H)$ so $\rho_{H'}(S)=\rho_H(S)\geq 0$. Now, consider a vertex set $T$, intersecting with $P$, with the minimum potential. Observe that vertices in $T$ have degree, in $H'[T]$, at least 2. Otherwise, it suffices to remove a vertex of degree 0 or 1 from $T$ and we obtain a set with lower potential (which contradicts the minimality of $\rho(T)$) as removing an isolated vertex decreases the potential by 9 and removing a vertex of degree 1 decreases the potential by $9-7=2$. Consequently, $T$ must contain the whole path $P$ as well as $u,v$. Observe that $T-P$ is a subset of $V(H)$ that contains $u,v$ and $\rho_{H'}(T)=\rho_H(T-P)+9k-7(k+1) = \rho_H(T-P) + 2k - 7 \geq \rho^*_H(\{u,v\}) + 2k - 7 \geq 7 - 2k + 2k - 7 =0$.

\end{proof}


\begin{observation} \label{pot observation}
Let $0\leq k\leq 3$, observe that in the proof of \Cref{potlemma}, adding a path $P$ between $u$ and $v$ in $H$ with $k$ internal 2-vertices decreases the potential of every set containing $\{u,v\}$ by at most $7-2k$ in $H+P$. In other words, for every $S\subseteq V(H+P)$ such that $\{u,v\}\subseteq S$, $\rho^*_{H+P}(S) \geq \rho^*_{H}(S-P)-(7-2k)$ ($S-P$ still contains $\{u,v\}$).

\end{observation}

\begin{lemma} \label{potlemma2}
Let $H$ be a graph, $u,v\in V(H)$ and $0\leq k\leq 3$. Let $H'=H+P$ be the graph obtained from $H$ in which we add a path $P$ with $k$ internal 2-vertices between $u$ and $v$ ($P$ is just an edge when $k=0$), then for every $A\subseteq V(H)$, $\rho^*_H(A)=\rho^*_{H'}(A)$ or $\rho^*_{H'}(A)\leq \rho^*_H(A)\leq \rho^*_H(A\cup\{u,v\})\leq \rho^*_{H'}(A) + (7-2k)$.
\end{lemma}

\begin{proof}
First, by \Cref{potsubgraph}, $\rho^*_H(A)\geq \rho^*_{H+P}(A)$.
Let us consider $S\subseteq V(H+P)$ such that $A\subseteq S$ and $\rho_{H+P}(S)=\rho_{H+P}^*(A)$. Note that $\rho_{H+P}(S)=\rho^*_{H+P}(S)$ since $\rho_{H+P}(S)\geq \rho^*_{H+P}(S)$ by definition of $\rho^*$, and $\rho_{H+P}(S)\leq \rho^*_{H+P}(S)$ or else, it means there exists $T$ such that $A\subseteq S \subset T$ and $\rho_{H+P}(T)<\rho_{H+P}(S)=\rho^*_{H+P}(A)$ which is a contradiction.  

If $\{u,v\}\not\subseteq S$, then $S\subseteq V(H)$ as $S$ cannot intersect $P$ by minimality of $\rho_{H+P}(S)$ (or else it would contain a vertex, of degree 0 or 1 in $H'[S]$, whose removal would decrease the potential). Thus, $\rho_H(S)=\rho_{H+P}(S)$. As a result, $\rho^*_H(A)\leq \rho_H(S)=\rho_{H+P}(S) = \rho_{H+P}^*(A)$.

If $\{u,v\}\subseteq S$, then $S$ contains $P$ by minimality of $\rho_{H+P}(S)$ (or else by adding $P$ to $S$, we would decrease the potential by $7(k+1)-9k=7-2k\geq 1$). By \Cref{pot observation}, $\rho_H^*(S-P)-(7-2k)\leq \rho^*_{H+P}(S)$. So, by \Cref{potsuperset}, $\rho^*_H(A)\leq \rho^*_H(A\cup\{u,v\}) \leq \rho^*_H(S-P)\leq \rho_{H+P}^*(S)+(7-2k) = \rho_{H+P}(S)+(7-2k) = \rho_{H+P}^*(A) +(7-2k)$.

\end{proof}

From now on, we will write $\rho^*_G(v_0v_1\dots v_i)$ instead of $\rho^*_G(\{v_0,v_1,\dots,v_i\})$ for conciseness. Also, for a graph $H$, we will say $\mad(H)\leq \frac{18}{7}$ instead of ``for all $S\subseteq V(H)$, $\rho^*_H(S)\geq 0$'' which is equivalent.


\section{Proof of \Cref{theorem}}
\label{sec3}

\paragraph{Notations and drawing conventions.} For $v\in V(G)$, the \emph{2-distance neighborhood} of $v$, denoted $N^*_G(v)$, is the set of 2-distance neighbors of $v$, which are vertices at distance at most two from $v$, not including $v$. We also denote $d^*_G(v)=|N^*_G(v)|$. We will drop the subscript and the argument when it is clear from the context. Also for conciseness, from now on, when we say ``to color'' a vertex, it means to color such vertex differently from all of its colored neighbors at distance at most two. Similarly, any considered coloring will be a 2-distance coloring.

Some more notations:
\begin{itemize}
\item A \emph{$d$-vertex} ($d^+$-vertex, $d^-$-vertex) is a vertex of degree $d$ (at least $d$, at most $d$). A $(\db{d}{e})$-vertex is a vertex with degree between $d$ and $e$ included.
\item A \emph{$k$-path} ($k^+$-path, $k^-$-path) is a path of length $k+1$ (at least $k+1$, at most $k+1$) where the $k$ internal vertices are 2-vertices.
\item A \emph{$(k_1,k_2,\dots,k_d)$-vertex} is a $d$-vertex incident to $d$ different paths, where the $i^{\rm th}$ path is a $k_i$-path for all $1\leq i\leq d$.
\end{itemize}
As a drawing convention for the rest of the figures, black vertices will have a fixed degree, which is represented, and white vertices may have a higher degree than what is drawn.


Let $G$ be a counterexample to \Cref{theorem} with the fewest number of vertices plus edges. Graph $G$ has maximum degree $\Delta=7$. The purpose of the proof is to prove that $G$ cannot exist. In the following we will study the structural properties of $G$ (\Cref{tutu}). We will then apply a discharging procedure (\Cref{tonton}). 

Since we have $\rho_G(V) = 9|V| - 7|E| \geq 0$, we must have 
\begin{equation}\label{equation}
\sum_{v\in V(G)} \left(7d(v)-18\right) \leq 0
\end{equation}

We assign to each vertex $v$ the charge $\mu(v)=7d(v)-18$. To prove the non-existence of $G$, we will redistribute the charges preserving their sum and obtaining a positive total charge, which will contradict \Cref{equation}.


\subsection{Structural properties of $G$\label{tutu}}

\begin{lemma}\label{connected}
Graph $G$ is connected.
\end{lemma}
Otherwise a component of $G$ would be a smaller counterexample.

\begin{lemma}\label{minimumDegree}
The minimum degree of $G$ is at least 2.
\end{lemma}
By \Cref{connected}, the minimum degree is at least 1. If $G$ contains a degree 1 vertex $v$, then we can simply remove the unique edge incident to $v$ and 2-distance color the resulting graph, which is possible by minimality of $G$. Then, we add the edge back and color $v$ (at most 7 constraints and 8 colors).

\medskip
\begin{lemma}\label{counting lemma}
Let $w$ be a vertex of $V$ that is adjacent to $k$  vertices $u_i$ $(k\le d(w))$, each satisfying $d^*(u_i)\le \Delta +i-1$ for $1\le i\le k$. Then we have $d^*(w)\ge \Delta+k+1$.
\end{lemma}

\begin{proof}
Suppose by contradiction that $w$ is adjacent to $u_i$ with $d^*_G(u_i)\leq \Delta+i-1$ for $1\leq i \leq k$, but $d^*_G(w)\leq \Delta+k$ (see \Cref{fig:01}). We remove the edges $wu_i$ for $1\le i \le k$. By minimality of $G$, let $\phi_H$ be a coloring of $H=(V,E\setminus\{wu_1,\ldots,wu_k\})$.

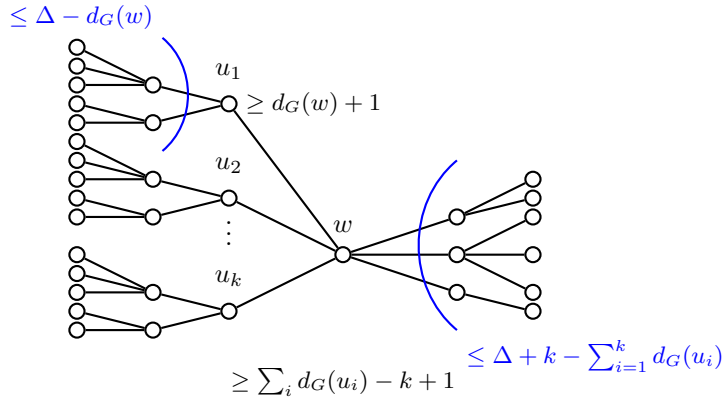
\begin{figure}[H]
\begin{center}
\begin{tikzpicture}[scale=0.5]
\begin{scope}[every node/.style={circle,thick,draw,minimum size=1pt,inner sep=2}]
    \node[label={above:$w$},label={below:\small $\geq \sum_i d_G(u_i)-k+1$}] (0) at (1,0) {};
    \node (10) at (4,1) {};
    \node (20) at (4,0) {};
    \node (30) at (4,-1) {};
    \node (11) at (6,2) {};
    \node (12) at (6,1.5) {};
    \node (21) at (6,1) {};
    \node (22) at (6,0) {};
    \node (23) at (6,-1) {};
    \node (31) at (6,-1.5) {};

    \node[label={above:$u_1$},label={right:\small $\geq d_G(w)+1$}] (100) at (-2,4) {};
    \node (110) at (-4,4.5) {};
    \node (120) at (-4,3.5) {};
    \node (111) at (-6,5.5) {};
    \node (112) at (-6,5) {};
    \node (113) at (-6,4.5) {};
    \node (121) at (-6,4) {};
    \node (122) at (-6,3.5) {};

    \node[label={above:$u_2$}] (200) at (-2,1.5) {};
    \node (210) at (-4,2) {};
    \node (220) at (-4,1) {};
    \node (211) at (-6,3) {};
    \node (212) at (-6,2.5) {};
    \node (213) at (-6,2) {};
    \node (221) at (-6,1.5) {};
    \node (222) at (-6,1) {};

    \node[label={above:$u_k$}] (300) at (-2,-1.5) {};
    \node (310) at (-4,-1) {};
    \node (320) at (-4,-2) {};
    \node (311) at (-6,0) {};
    \node (312) at (-6,-0.5) {};
    \node (313) at (-6,-1) {};
    \node (321) at (-6,-1.5) {};
    \node (322) at (-6,-2) {};
\end{scope}

\begin{scope}[every edge/.style={draw=black,thick}]
    \path (0) edge (10);
    \path (0) edge (20);
    \path (0) edge (30);
    \path (10) edge (11);
    \path (10) edge (12);
      \path (20) edge (21);
      \path (20) edge (22);
      \path (20) edge (23);
      \path (30) edge (31);
        \path (0) edge (100);
      \path (100) edge (110);
      \path (100) edge (120);
      \path (110) edge (111);
      \path (110) edge (112);
      \path (110) edge (113);
      \path (120) edge (121);
      \path (120) edge (122);
        \path (0) edge (200);
      \path (200) edge (210);
      \path (200) edge (220);
      \path (210) edge (211);
      \path (210) edge (212);
      \path (210) edge (213);
      \path (220) edge (221);
      \path (220) edge (222);
        \path (0) edge (300);
      \path (300) edge (310);
      \path (300) edge (320);
      \path (310) edge (311);
      \path (310) edge (312);
      \path (310) edge (313);
      \path (320) edge (321);
      \path (320) edge (322);
\end{scope}

\coordinate[label={below right:\small \textcolor{blue}{$\leq \Delta+k-\sum_{i=1}^k d_G(u_i)$}}] (A) at (4,-2);
\coordinate (B) at (4,2.5);
\coordinate[label={above left:\small \textcolor{blue}{$\leq \Delta-d_G(w)$}}] (C) at (-3.75,5.75);
\coordinate (D) at (-3.75,2.75);
\coordinate[label={above:$\vdots$}] (G) at (-2,0);

\draw[thick,blue] (A) to[bend left=50] (B);
\draw[thick,blue] (C) to[bend left=50] (D);
\end{tikzpicture}
\vspace{-1.5cm}
\caption{The configuration of~\Cref{counting lemma}.\label{fig:01}} 
\end{center}
\end{figure}


We extend $\phi$ to $G$ as follows :
\begin{enumerate}
\item We define $\phi_G(v)=\phi_H(v)$ for all $v\in V\setminus\{w,u_1,\ldots,u_k\}$.
\item We choose $\phi_G(w)$ a color that does not appear in $F_w = N^*_G(w)\setminus\{u_1,\dots,u_k\}$. We have $|F_w| = d^*_G(w) - k$. By hypothesis, we have $d^*_G(w)\leq \Delta+k$ and thus $|F_w| \le \Delta$. Thus, we can color $w$ since we have $\Delta+1$ colors.
\item One by one, from $k$ to 1, we choose $\phi_G(u_i)$ a color that does not appear in $F_{u_i} = N^*_G(u_i)\setminus\{u_1,\dots,u_{i-1}\}$. Since $d^*_G(u_i)\leq \Delta+i-1$, we have $|F_{u_i}| = d^*_G(u_i) - (i-1) \le (\Delta +i-1)-(i-1) = \Delta$. So, there remains at least one color for $u_i$.
\end{enumerate}

Observe that we 2-distance colored the vertices $w$, $u_1, \ldots, u_k$. Hence the obtained coloring $\phi_G$ is a 2-distance coloring of $G$, which is a contradiction.
\end{proof}

\begin{observation}\label{counting observation}
Let $w$ be a vertex of $V$ that is adjacent to $k$  vertices $u_i$, each satisfying $d^*(u_i)\le 7 = \Delta$ ($\le \Delta+i-1$) for $1\le i\le k$. Then we have $d^*(w)\ge \Delta+k+1 = 8+k$ due to \Cref{counting lemma}.
\end{observation}


\begin{lemma} \label{4-path lemma}
Graph $G$ has no $4^+$-paths.
\end{lemma}

\begin{proof}
Suppose $G$ contains a 4-path $v_0v_1\dots v_5$ (see \Cref{4-path lemma figure}). Then $d^*(v_2)=d^*(v_3)=4\leq \Delta$ which contradicts \Cref{counting observation}.
\end{proof}
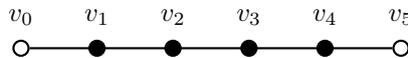
\begin{figure}[H]
\begin{center}
\begin{tikzpicture}{thick}
\begin{scope}[every node/.style={circle,thick,draw,minimum size=1pt,inner sep=2}]
    \node[label={above:$v_5$}] (0) at (0,0) {};

    \node[label={above:$v_0$}] (100) at (-5,0) {};

    \node[fill,label={above:$v_4$}] (1) at (-1,0) {};
    \node[fill,label={above:$v_3$}] (2) at (-2,0) {};
    \node[fill,label={above:$v_2$}] (3) at (-3,0) {};
    \node[fill,label={above:$v_1$}] (4) at (-4,0) {};
\end{scope}

\begin{scope}[every edge/.style={draw=black,thick}]
    \path (0) edge (100);
\end{scope}
\end{tikzpicture}
\caption{A 4-path.}
\label{4-path lemma figure}
\end{center}
\end{figure}

\begin{lemma}\label{3-path lemma}
A $3$-path has two distinct endvertices and both have degree $\Delta$.
\end{lemma}

\begin{proof}
Suppose that $G$ contains a 3-path $v_0v_1\dots v_4$ (see \Cref{3-path lemma figure}). 

If $v_0=v_4$, then we color $H=G-\{v_1,v_2,v_3\}$ by minimality of $G$ and extend the coloring to $G$ by coloring greedily $v_1$ and $v_3$ who has two available colors each and finish with $v_2$ who only sees three colors.

Now, suppose that $v_0\neq v_4$, since $d^*(v_2)=4\leq \Delta$, we have $d^*(v_3)\geq \Delta+2$ due to \Cref{counting observation}. Moreover, $d^*(v_3)=d(v_4)+2$, so $d(v_4)\geq \Delta$. The same holds for $v_0$ by symmetry.  
\end{proof}
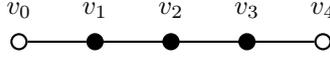
\begin{figure}[H]
\begin{center}
\begin{tikzpicture}{thick}
\begin{scope}[every node/.style={circle,thick,draw,minimum size=1pt,inner sep=2}]
    \node[label={above:$v_4$}] (0) at (-1,0) {};

    \node[label={above:$v_0$}] (100) at (-5,0) {};

    \node[fill,label={above:$v_3$}] (2) at (-2,0) {};
    \node[fill,label={above:$v_2$}] (3) at (-3,0) {};
    \node[fill,label={above:$v_1$}] (4) at (-4,0) {};
\end{scope}

\begin{scope}[every edge/.style={draw=black,thick}]
    \path (4) edge (100);
    \path (4) edge (0);
\end{scope}

\end{tikzpicture}
\end{center}
\caption{A 3-path.}
\label{3-path lemma figure}
\end{figure}

\begin{lemma} \label{2-path lemma}
At least one of the endvertices of a $2$-path has degree $\Delta$ or both of them have degree $\Delta-1$. The endvertices of a $2$-path are also distinct.
\end{lemma}

\begin{proof}
Consider a 2-path $v_0v_1v_2v_3$ (see \Cref{2-path lemma figure}) where $d(v_0)\leq d(v_3)$. 

If $v_0=v_3$, then we color $H=G-\{v_1,v_2\}$ by minimality of $G$ and extend the coloring to $G$ by coloring greedily $v_1$ and $v_2$ who has two available colors each.

Now, suppose that $v_0\neq v_3$. Suppose by contradiction that $d(v_3)\leq \Delta-1$ and $d(v_0)\leq \Delta-2$. Since $d(v_0)\leq \Delta-2$, $d^*(v_1)=d(v_0)+2\leq \Delta$. So, by \Cref{counting observation}, $d^*(v_2)=d(v_3)+2\geq \Delta+2$ meaning that $d(v_3)\geq \Delta$, which is a contradiction.

\end{proof}

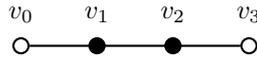
\begin{figure}[H]
\begin{center}
\begin{tikzpicture}{thick}
\begin{scope}[every node/.style={circle,thick,draw,minimum size=1pt,inner sep=2}]
    \node[label={above:$v_3$}] (0) at (-1,0) {};

    \node[label={above:$v_0$}] (100) at (-4,0) {};

    \node[fill,label={above:$v_2$}] (2) at (-2,0) {};
    \node[fill,label={above:$v_1$}] (3) at (-3,0) {};
\end{scope}

\begin{scope}[every edge/.style={draw=black,thick}]
    \path (0) edge (100);
\end{scope}
\end{tikzpicture}
\end{center}
\caption{A 2-path.}
\label{2-path lemma figure}
\end{figure}

\begin{lemma} \label{2-path lemma1}
Let $uvwx$ be a $2$-path. If $d(u)=7$ and $d(x)\leq 6$, then $u$ cannot be adjacent to $x$.
\end{lemma}

\begin{proof}
Suppose by contradiction that $u$ is adjacent to $x$. Let $H=G-\{v,w\}$. By minimality of $G$, we color $H$, then we finish by coloring $v$ then $w$ in this order.
\end{proof}

\begin{lemma} \label{tree lemma}
Graph $G$ has no cycles consisting of $3$-paths.
\end{lemma}

\begin{proof}
Suppose that $G$ contains a cycle consisting of $k$ 3-paths (see \Cref{tree lemma figure}). We remove all vertices $v_{4i+1}$, $v_{4i+2}$, $v_{4i+3}$ for $0\leq i\leq k-1$. Consider a coloring of the resulting graph. We color greedily $v_1, v_3, v_5, \dots, v_{4k-1}$. This is possible since each of them has at least two choices of colors (as $d(v_0)=d(v_4)=\dots=d(v_{4(k-1)})=\Delta$ due to \Cref{3-path lemma}) and by 2-choosability of even cycles. Finally, it is easy to color greedily $v_2,v_6,\dots,v_{4k-2}$ since they each have at most four forbidden colors.
\end{proof}
\begin{figure}[H]
\begin{center}
\begin{tikzpicture}[scale=0.7,rotate=90]
\begin{scope}[every node/.style={circle,thick,draw,minimum size=1pt,inner sep=2}]
    \node[label={above:$v_4$}] (0) at (-1,-1) {};

    \node[label={below:$v_0$}] (100) at (-5,-1) {};

    \node[fill,label={left:$v_3$},label={right:$2$}] (1) at (-2,0) {};
    \node[fill,label={left:$v_2$},label={right:$\Delta-3$}] (2) at (-3,0) {};
    \node[fill,label={left:$v_1$},label={right:$2$}] (3) at (-4,0) {};

    \node[fill,label={below:$2$}] (4) at (-1,-2) {};
    \node[draw=none] (5) at (-1,-3) {};
    \node[fill] (6) at (-1,-4) {};

    \node[label={above:$v_{4j}$}] (200) at (-1,-5) {};

    \node[fill] (7) at (-2,-6) {};
    \node[fill] (8) at (-3,-6) {};
    \node[fill] (9) at (-4,-6) {};

    \node[label={below:$v_{4(j+1)}$}] (300) at (-5,-5) {};

    \node[fill,label={below:$v_{4k-1}$},label={above:$2$}] (14) at (-5,-2) {};
    \node[draw=none] (15) at (-5,-3) {};
    \node[fill] (16) at (-5,-4) {};
\end{scope}

\begin{scope}[every edge/.style={draw=black,thick}]
    \path (0) edge (1);
    \path (1) edge (3);
    \path (3) edge (100);
        \path (0) edge (4);
      \path[dashed] (4) edge (5);
      \path[dashed] (5) edge (6);
      \path (6) edge (200);
          \path (200) edge (7);
      \path (7) edge (9);
      \path (9) edge (300);
          \path (100) edge (14);
      \path[dashed] (14) edge (15);
      \path[dashed] (15) edge (16);
      \path (16) edge (300);
\end{scope}
\end{tikzpicture}
\vspace{-0.5cm}
\caption{A cycle consisting of consecutive 3-paths.}
\label{tree lemma figure}
\end{center}
\end{figure}
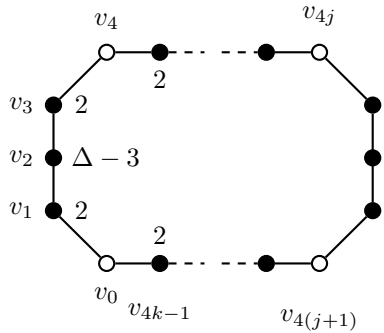

\begin{lemma} \label{small vertex lemma}
Let $v\in V$ such that $3\leq d(v)\leq \lfloor\frac{\Delta+1}{2}\rfloor$. Then $v$ cannot be a $(2,1^+,1^+,\dots,1^+)$-vertex.
\end{lemma}

\begin{proof}
Suppose that $G$ contains a vertex $v$ with $3\le d(v) \le \lfloor\frac{\Delta+1}{2}\rfloor$ that is a $(2,1^+,1^+,\dots,1^+)$-vertex. Let $w$ be a neighbor of $v$ that belongs to a 2-path. See \Cref{small vertex lemma figure}. We have $d^*(w)=d(v)+2$ and $d^*(v)=2d(v)$. Moreover, as $d(v)\leq \lfloor\frac{\Delta+1}{2}\rfloor$, it follows that $d^*(w)\leq \Delta$ since $\Delta > 3$. Thus, $d^*(v)\geq \Delta+2$ by \Cref{counting observation}. Since $d(v)$ is an integer and $2d(v)\geq \Delta+2$, $d(v)\geq \lceil\frac{\Delta+2}{2}\rceil$ which contradicts $d(v)\leq \lfloor\frac{\Delta+1}{2}\rfloor$.
\end{proof}
\begin{figure}[H]
\begin{center}
\begin{tikzpicture}{thick}
\begin{scope}[every node/.style={circle,thick,draw,minimum size=1pt,inner sep=2}]
    \node[fill,label={above:$v$}] (0) at (-2,0) {};
    \node[fill] (10) at (-1,1) {};
    \node (11) at (0,1) {};
    \node[fill] (20) at (-1,0) {};
    \node (21) at (0,0) {};
    \node[fill] (30) at (-1,-1) {};
    \node (31) at (0,-1) {};
    \node (100) at (-5,0) {};

    \node[fill,label={above:$w$}] (3) at (-3,0) {};
    \node[fill] (4) at (-4,0) {};
\end{scope}

\begin{scope}[every edge/.style={draw=black,thick}]
    \path (0) edge (100);
    \path (0) edge (10);
    \path (0) edge (20);
    \path (0) edge (30);
      \path (10) edge (11);
      \path (20) edge (21);
      \path (30) edge (31);
\end{scope}
\end{tikzpicture}
\caption{A $(2,1^+,\ldots,1^+)$-vertex $v$ with $3\le d(v) \le \lfloor \frac{\Delta+1}{2} \rfloor$. }
\label{small vertex lemma figure}
\end{center}
\end{figure}
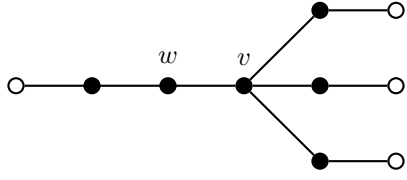

\begin{observation} \label{Hall} 
Let $u_1,u_2,\dots,u_k$ be $k$ vertices that are pairwise at distance at most two and let $L_i$ be the list of available colors of $u_i$ for $1\leq i\leq k$. By Hall's Theorem, if for all $1\leq l \leq k$, for all $i_1,i_2,\dots,i_l$, $|\cup_{j=1}^l L_{i_j}|\geq l$, then $u_1,u_2,\dots,u_k$ are colorable with the lists $L_1,L_2\dots,L_k$.  
\end{observation}

\begin{lemma} \label{weird cases lemma}
Let $u$ be a $7$-vertex that is incident to six $2$-paths where the other endvertices are $5^-$-vertices. Then, $u$ cannot be incident to a $3$-path, a $(2,2,0)$-vertex or another $2$-path where the other endvertex is a $6^-$-vertex. 
\end{lemma}

\begin{proof}
Suppose by contradiction that $u$ is incident to six $2$-paths where the other endvertices are $5^-$-vertices and that $u$ is also to a $3$-path, a $(2,2,0)$-vertex or another $2$-path where the other endvertex is a $6^-$-vertex.
 
First, observe that $u$ is distinct from the other endvertex of its incident $3$-path due to \Cref{3-path lemma} and from the endvertices of the $2$-paths incident to its $(2,2,0)$-neighbor due to \Cref{2-path lemma1}.

Consider $H=G-(\{u\}\cup N^*_G(u))$. By minimality of $G$, there exists a coloring of $H$ that we will extend to $G$ by coloring the vertices in the order indicated in \Cref{weird cases lemma figure} with the specification that in the case where $u$ is incident to another $2$-path with a $6^-$-endvertex, $u$'s $2$-neighbor $x$ on this path will be colored with the same color as a colored vertices at distance 3 from $x$. The indicated order verifies at each step that the considered vertex sees at most seven colors. Thus, we obtain a valid coloring of $G$ which is a contradiction.
\end{proof}

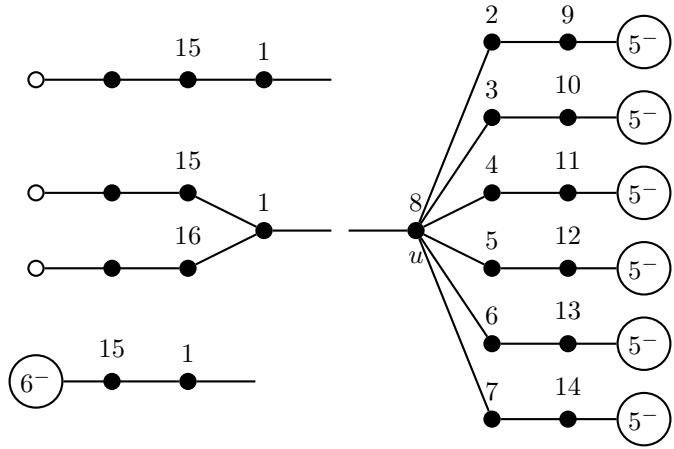
\begin{figure}[H]
\begin{center}
\begin{tikzpicture}{thick}
\begin{scope}[every node/.style={circle,thick,draw,minimum size=1pt,inner sep=2}]
    \node[fill,label={below:$u$},label={above:8}] (0) at (-2,0) {};
    \node[fill,label={above:2}] (10) at (-1,2.5) {};
    \node[fill,label={above:9}] (11) at (0,2.5) {};
    \node (12) at (1,2.5) {$5^-$};
    \node[fill,label={above:3}] (20) at (-1,1.5) {};
    \node[fill,label={above:10}] (21) at (0,1.5) {};
    \node (22) at (1,1.5) {$5^-$};
    \node[fill,label={above:4}] (30) at (-1,0.5) {};
    \node[fill,label={above:11}] (31) at (0,0.5) {};
    \node (32) at (1,0.5) {$5^-$};
    \node[fill,label={above:5}] (40) at (-1,-0.5) {};
    \node[fill,label={above:12}] (41) at (0,-0.5) {};
    \node (42) at (1,-0.5) {$5^-$};
    \node[fill,label={above:6}] (50) at (-1,-1.5) {};
    \node[fill,label={above:13}] (51) at (0,-1.5) {};
    \node (52) at (1,-1.5) {$5^-$};
    \node[fill,label={above:7}] (60) at (-1,-2.5) {};
    \node[fill,label={above:14}] (61) at (0,-2.5) {};
    \node (62) at (1,-2.5) {$5^-$};

    \node[draw=none] (1) at (-3,0) {};
    
    \node[draw=none] (2) at (-3,2) {};
    \node[fill,label={above:1}] (3) at (-4,2) {};
    \node[fill,label={above:15}] (4) at (-5,2) {};
    \node[fill] (5) at (-6,2) {};
    \node (6) at (-7,2) {};
    
    \node[fill,label={above:1}] (13) at (-4,0) {};
    \node[fill,label={above:15}] (14) at (-5,0.5) {};
    \node[fill] (15) at (-6,0.5) {};
    \node (16) at (-7,0.5) {};
    
    \node[fill,label={above:16}] (17) at (-5,-0.5) {};
    \node[fill] (18) at (-6,-0.5) {};
    \node (19) at (-7,-0.5) {};
    
    \node[draw=none] (23) at (-4,-2) {};
    \node[fill,label={above:1}] (24) at (-5,-2) {};
    \node[fill,label={above:15}] (25) at (-6,-2) {};
    \node (26) at (-7,-2) {$6^-$};
\end{scope}

\begin{scope}[every edge/.style={draw=black,thick}]
	\path (0) edge (1);
    \path (0) edge (10);
    \path (0) edge (20);
    \path (0) edge (30);
    \path (0) edge (40);
    \path (0) edge (50);
    \path (0) edge (60);
      \path (10) edge (12);
      \path (20) edge (22);
      \path (30) edge (32);
      \path (40) edge (42);
      \path (50) edge (52);
      \path (60) edge (62);
    
    \path (2) edge (3);
    \path (3) edge (6);
    
    \path (1) edge (13);
    \path (13) edge (14);
    \path (13) edge (17);
    \path (14) edge (16);
    \path (17) edge (19);
    
    \path (23) edge (24);
    \path (24) edge (26);
\end{scope}
\end{tikzpicture}
\caption{The order in which the coloring will be extended to $G$ is indicated above the vertices.}
\label{weird cases lemma figure}
\end{center}
\end{figure}

\begin{lemma} \label{weird cases lemma2}
A $6$-vertex cannot be incident to six $2$-paths where the other endvertices are $6$-vertices.
\end{lemma}

\begin{proof}
Consider $H$ the graph $G$ where we removed the $6$-vertex $u$ and all $2$-vertices on the $2$-paths incident to $u$. Consider the internal $2$-vertices $p_1$ and $p_2$ on a $2$-path incident to $u$. We color the $2$-vertex $p_2$ at distance 2 from $u$, which has at least two available colors, and a $2$-neighbor $x$ of $u$, which has seven available colors, with the same color by the pigeonhole principle. Now, we color all other $2$-vertices at distance 2 from $u$, then $u$. Finally, we color all vertices of $N_G(u)$ by finishing with $p_1$ which now sees eight colored vertices but two of them share the same color.
\end{proof}

\begin{lemma} \label{pot 3-path lemma}
Let $1\leq k\leq 3$ and $up_1\dots p_kv$ be a $k$-path in $G$ and let $P=\{p_1,\dots,p_k\}$. If $\rho^*_{G-P}(u)\leq \rho^*_{G-P}(v)$, then $\rho^*_{G-P}(v)\geq 1$.
\end{lemma}

\begin{proof} 
Suppose by contradiction that $\rho^*_{G-P}(v)=\rho^*_{G-P}(u)=0$ (recall that $\rho^*_{G-P}(u)\geq 0$ since $\mad(G)\leq \frac{18}{7}$). Then, by \Cref{potadd}, $0=\rho^*_{G-P}(v)+\rho^*_{G-P}(u)\geq \rho^*_{G-P}(uv)$. 
However, by \Cref{potlowerbound}, $\rho^*_{G-P}(uv)\geq 7|E(P,\{u,v\})| - \rho_G(P) \geq 7\cdot 2 - (9\cdot 3 - 7\cdot 2) = 1$, which is a contradiction.
\end{proof}

\begin{lemma}\label{two 3-paths lemma}
Let $up_1p_2p_3v$ and $vp'_1p'_2p'_3w$ be two consecutive $3$-paths in $G$ and let $P=\{p_1,p_2,p_3\}$ and $P'=\{p'_1,p'_2,p'_3\}$. Then $\rho^*_{G-P}(u) = \rho^*_{G-P'}(w) = 0$.
\end{lemma}

\begin{proof}
Note that by \Cref{3-path lemma} and \Cref{tree lemma}, $u$, $v$ and $w$ are pairwise distinct and $d(u)=d(v)=d(w)=\Delta=7$. Let $H=G-(P\cup P')$. We add the 3-path $up''_1p''_2p''_3w$ in $H$ and let $P''=\{p''_1,p''_2,p''_3\}$ and let $H+P''$ be the resulting graph. 

Suppose that $\rho^*_H(uw)\geq 1$. Then, by \Cref{potlemma} with $k=3$, $\mad(H+P'')\leq \frac{18}{7}$. Observe that $|V(H+P'')|+|E(H+P'')|<|V(G)|+|E(G)|$, so $H+P''$ is colorable with a coloring $\psi$ by minimality of $G$. We define $\phi$ a coloring of $G$ as follows:
\begin{itemize}
\item If $x\in V(H)$, then $\phi(x) = \psi(x)$.
\item Let $\phi(p_1) = \psi(p''_1)$ and $\phi(p'_3)=\psi(p''_3)$.
\item Observe that $p_3$ and $p'_1$ can be colored. Otherwise, they have to see the same seven colors at distance at most 2. Since they see the same 6 colored vertices in $v\cup N_G(v)\setminus\{p_3,p'_1\}$ (as $d(v)=\Delta=7$ by \Cref{3-path lemma}), $\phi(p_1)$ must be the same as $\phi(p'_3)$, which is impossible because $\phi(p_1) = \psi(p''_1) \neq \psi(p''_3)=\phi(p'_3)$.
\item Finally, $p_2$ and $p'_2$ can be colored greedily since they see at most 4 different colors at distance 2 each.
\end{itemize}
As $\psi$ is a valid coloring of $H+P''$, $\phi$ is a valid coloring of $G$, which is a contradiction.

Suppose that $\rho^*_H(uw) = 0$ (recall that $\rho^*_H(uw)\geq 0$ since $H$ is a subgraph of $G$). By \Cref{potsuperset}, $\rho^*_H(uw)\geq \rho^*_H(u)$ and by \Cref{potsubgraph}, $\rho^*_H(u)=\rho^*_{G-(P\cup P')}(u)\geq \rho^*_{G-P}(u)$. Hence, $0\leq \rho^*_{G-P}(u) \leq \rho^*_H(uw) = 0$. Symmetrically, the same holds for $\rho^*_{G-P'}(w)$.
\end{proof}

\begin{figure}[H]
\begin{center}
\begin{tikzpicture}{thick}
\begin{scope}[every node/.style={circle,thick,draw,minimum size=1pt,inner sep=2}]
    \node[fill,label={above:$p'_1$}] (0) at (0,0) {};

    \node[label={above:$u$}] (100) at (-5,0) {};

    \node[label={above:$v$}] (1) at (-1,0) {};
    \node[fill,label={above:$p_3$}] (2) at (-2,0) {};
    \node[fill,label={above:$p_2$}] (3) at (-3,0) {};
    \node[fill,label={above:$p_1$}] (4) at (-4,0) {};
    \node[fill,label={above:$p'_2$}] (5) at (1,0) {};
    \node[fill,label={above:$p'_3$}] (6) at (2,0) {};
    \node[label={above:$w$}] (7) at (3,0) {};
    
    \node[fill,label={above:$p''_1$}] (8) at (-3,-0.9) {};
    \node[fill,label={above:$p''_2$}] (9) at (-1,-1.2) {};
    \node[fill,label={above:$p''_3$}] (10) at (1,-0.9) {};
\end{scope}

\begin{scope}[every edge/.style={draw=black,thick}]
    \path (1) edge (100);
    \path (1) edge (7);
    \path (100) edge[bend right,dashed] (7);
\end{scope}
\end{tikzpicture}
\caption{Two consecutives $3$-paths.}
\label{two 3-paths lemma figure}
\end{center}
\end{figure}
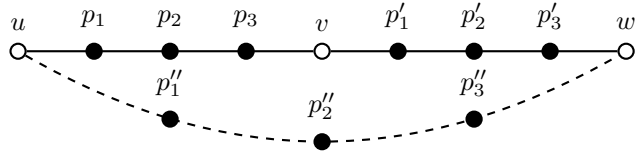

\begin{lemma} \label{three 3-paths lemma}
Graph $G$ has no three consecutive $3$-paths.
\end{lemma}

\begin{proof}
Suppose by contradiction that $G$ has three consecutive 3-paths $up_1p_2p_3v$, $vp'_1p'_2p'_3w$, and $wp''_1p''_2p''_3x$. Let $P'=\{p'_1p'_2p'_3\}$. By applying \Cref{two 3-paths lemma} to $up_1p_2p_3v$ and $vp'_1p'_2p'_3w$, we get $\rho^*_{G-P'}(w)=0$. By applying \Cref{two 3-paths lemma} to $vp'_1p'_2p'_3w$ and $wp''_1p''_2p''_3x$, we get $\rho^*_{G-P'}(v)=0$. This is impossible due to \Cref{pot 3-path lemma}.
\end{proof}

\begin{lemma} \label{weird cases lemma3}
Let $u$ be a $7$-vertex and $up_1p_2v$ be a $2$-path incident to $u$ and let $P=\{p_1,p_2\}$. If $\rho^*_{G-P}(u)\leq \rho^*_{G-P}(v)$, then $u$ cannot be incident to six other $2$-paths where the other endvertices are $5^-$-vertices.
\end{lemma}

\begin{proof}
Suppose by contradiction that $u$ is incident to seven $2$-paths, say $uq_iq'_iv_i$ for $1\leq i\leq 6$ and $up_1p_2v$. Note that by \Cref{2-path lemma}, $u$ is distinct from $v$. Let $H= G-(\{q_i,q'_i|1\leq i\leq 6\}\cup \{p_1,p_2\})$.

We claim that:
\begin{claim} \label{claim weird cases lemma3}
For all $1\leq i\leq 6$, $\rho^*_H(vv_i)\leq 2$.
\end{claim}

\begin{proof}
W.l.o.g. suppose by contradiction that $\rho^*_H(vv_1)\geq 3$. We add the $2$-path $vp'_1p'_2v_1$ in $H$ and let $P'=\{p'_1,p'_2\}$ and let $H+P'$ be the resulting graph. Since $\rho^*_H(vv_1)\geq 3$, by \Cref{potlemma} with $k=2$, $\mad(H+P')\leq \frac{18}{7}$. By minimality of $G$, there exists a coloring $\psi$ of $H+P'$. We define $\phi$ a coloring of $G$ as follows:
\begin{itemize}
\item If $x\in V(H)$, then $\phi(x)=\psi(x)$.
\item Let $\phi(p_2)=\psi(p'_1)$.
\item Note that $d^*_G(q'_i)\leq 7$ since $d_G(v_i)\leq 5$ for all $1\leq i\leq 6$. As a result, we can always color them last.
\item Let $L(x)$ be the list of available colors left for a vertex $x$. Observe that we have $|L(p_1)|\geq 6$ and $|L(u)|,|L(q_1)|,\dots,|L(q_6)|\geq 7$. By \Cref{Hall}, the only way these eight vertices are not colorable is if $|L(p_1)\cup L(u)\cup L(q_1)\cup\dots\cup L(q_6)|\leq 7$. As a result, $|L(u)\cup L(q_1)| = 7$ and $|L(u)|, |L(q_1)|\geq 7$. In other words, $L(u)=L(q_1)$ . However, $u$ sees $\phi(p_2)=\psi(p'_1)\neq \psi(v_1)=\phi(v_1)$ which $q_1$ sees. So $L(u)\neq L(q_1)$.
\end{itemize}
We obtain a valid coloring of $G$ which is a contradiction.
\end{proof}

First, recall that $\rho^*_{G-P}(u)\leq \rho^*_{G-P}(v)$ and by \Cref{pot 3-path lemma}, $\rho^*_{G-P}(v)\geq 1$. As a result, $\rho^*_H(v)\geq \rho^*_{G-P}(v)\geq 1$ by \Cref{potsubgraph}.

Now, by \Cref{claim weird cases lemma3}, we have $\sum_{i=1}^6\rho^*_H(vv_i)\leq 2\cdot 6=12$. However, by \Cref{potadd}, then \Cref{potlowerbound} and the fact that $\rho^*_H(v)\geq 1$, we have $\sum_{i=1}^6\rho^*_H(vv_i)\geq \rho^*_H(vv_1\dots v_6) + 5\rho^*_H(v)\geq 12 + 5\cdot 1 = 17$.

\end{proof}

\begin{figure}[H]
\begin{center}
\begin{tikzpicture}{thick}
\begin{scope}[every node/.style={circle,thick,draw,minimum size=1pt,inner sep=2}]
    \node[label={above:$u$}] (0) at (-2,0) {};
    \node[fill,label={above:$q_1$}] (30) at (-1,1) {};
    \node[fill,label={above:$q'_1$}] (31) at (0,1) {};
    \node[label={above:$v_1$}] (32) at (1,1) {$5^-$};
    \node[fill,label={above:$q_6$}] (40) at (-1,-1) {};
    \node[fill,label={above:$q'_6$}] (41) at (0,-1) {};
    \node[label={above:$v_6$}] (42) at (1,-1) {$5^-$};

    \node[fill,label={above:$p_1$}] (1) at (-3,0) {};
    \node[fill,label={above:$p_2$}] (2) at (-4,0) {};
    \node[label={above:$v$}] (3) at (-5,0) {};
    
    \node[draw=none] (4) at (-0.5,0) {$\dots$};
    
    \node[fill,label={above:$p'_1$}] (5) at (-3,2) {};
    \node[fill,label={above:$p'_2$}] (6) at (-1,2.2) {};
\end{scope}

\begin{scope}[every edge/.style={draw=black,thick}]
	\path (0) edge (3);
    \path (0) edge (30);
    \path (0) edge (40);
      \path (30) edge (32);
      \path (40) edge (42);
      
   \path (3) edge[bend left=60,dashed] (32);
\end{scope}
\end{tikzpicture}
\caption{A $7$-vertex incident to seven $2$-paths, six of which have $5^-$-endvertices.}
\label{weird cases lemma3 figure}
\end{center}
\end{figure}
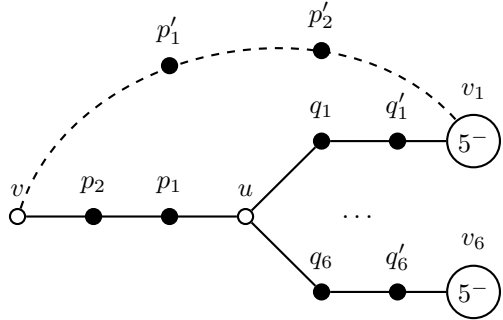

\begin{lemma} \label{sponsor lemma0}
Consider $u$ a $7$-vertex that is incident to a unique $3$-path $up_1p_2p_3v$ and let $P=\{p_1,p_2,p_3\}$. If $\rho^*_{G-P}(u)\leq \rho^*_{G-P}(v)$, then $u$ is incident to at most two $2$-paths where the other endvertices are $5^-$-vertices.
\end{lemma}

\begin{proof}
Note that by \Cref{3-path lemma}, $u$ and $v$ are distinct vertices and $d(u)=d(v)=7$. Suppose by contradiction that $u$ is incident to at least three 2-paths $uq_1q'_1v_1$, $uq_2q'_2v_2$, and $uq_3q'_3v_3$ where $v_1$, $v_2$, and $v_3$ have degree at most 5. Note that by \Cref{2-path lemma}, $u$ is distinct from $v_i$ for all $1\leq i\leq 3$.

Let $H=G-\{p_1,p_2,p_3,q_1,q'_1,q_2,q'_2,q_3,q'_3\}$. Recall that $\rho^*_{G-P}(u)\leq \rho^*_{G-P}(v)$ and by \Cref{pot 3-path lemma}, $\rho^*_{G-P}(v)\geq 1$. As a result, $\rho^*_H(v)\geq \rho^*_{G-P}(v)\geq 1$ by \Cref{potsubgraph}.

We claim the following:
\begin{claim} \label{claim sponsor lemma0}
For all $1\leq i\leq 3$, $\rho^*_H(vv_i)\leq 2$ and $\rho^*_H(uv)\leq 6$.  
\end{claim}

\begin{proof}
Let us prove \Cref{claim sponsor lemma0} by contradiction.

First, suppose w.l.o.g. that $\rho^*_H(vv_1)\geq 3$. We add the 2-path $vp'_1p'_2v_1$ in $H$, let $P'=\{p'_1,p'_2\}$, and let $H+P'$ be the resulting graph. Since $\rho^*_H(vv_1)\geq 3$, by \Cref{potlemma} with $k=2$, $\mad(H+P')\leq \frac{18}7$. By minimality of $G$, there exists a coloring $\psi$ of $H+P'$. We define $\phi$ a coloring of $G$ as follows:
\begin{itemize}
\item If $x\in V(H)$, then $\phi(x)=\psi(x)$.
\item Let $\phi(p_3)=\psi(p'_1)$.
\item Note that $d^*_G(q'_i)\leq \Delta = 7$ since $d_G(v_i)\leq 5$ for all $1\leq i\leq 3$. As a result, we can always color them last. The same holds for $p_2$ since $d^*_G(p_2)=4$.
\item The only vertices left uncolored are $p_1,q_1,q_2,q_3$ and each of them has at least three available colors left. By \Cref{Hall}, they can be colored unless they have exactly the same three available colors each. Since they all see the same fours colors in $N_H(u)\cup\{u\}$ and $\phi(p_3)=\psi(p'_1)\neq \psi(v_1)=\phi(v_1)$, $p_1$ and $q_1$ cannot have the same three available colors.   
\end{itemize}

We obtain a valid coloring $\phi$ of $G$ so $\rho^*_H(vv_i)\leq 2$ for all $1\leq i\leq 3$.

Now, suppose that $\rho^*_H(uv)\geq 7$. We add the edge $e=uv$ in $H$. Since $\rho^*_H(uv)\geq 7$, by \Cref{potlemma} with $k=0$, $\mad(H+e)\leq \frac{18}7$. By minimality of $G$, there exists a coloring $\psi$ of $H+e$. We define $\phi$ a coloring of $G$ as follows:
 \begin{itemize}
\item If $x\in V(H)$, then $\phi(x)=\psi(x)$.
\item Let $\phi(p_3)=\psi(u)$.
\item Note that $d^*_G(q'_i)\leq \Delta = 7$ since $d_G(v_i)\leq 5$ for all $1\leq i\leq 3$. As a result, we can always color them last. The same holds for $p_2$ since $d^*_G(p_2)=4$.
\item We color $q_1,q_2,q_3$ which is possible since they have at least three available colors each. 
\item We finish by coloring $p_1$ which sees eight colored vertices but since it sees $\phi(u)=\phi(p_3)$ twice, it has at least one available color left.
\end{itemize}

We obtain a valid coloring $\phi$ of $G$ so $\rho^*_H(uv)\leq 6$. Thus, \Cref{claim sponsor lemma0} is true.
\end{proof}

By \Cref{claim sponsor lemma0}, we get $\rho_H^*(uv)+\sum^3_{i=1}\rho^*_H(vv_i) \leq 6 + 3\cdot 2 = 12$. However, by \Cref{potadd} then by \Cref{potlowerbound} and recall that $\rho^*_H(v)\geq 1$, we get $\rho_H^*(uv)+\sum^3_{i=1}\rho^*_H(vv_i) \geq \rho^*_H(uvv_1v_2v_3) + 3\rho^*_H(v) \geq 10 + 3\cdot 1 = 13$ which is a contradiction.

\end{proof}

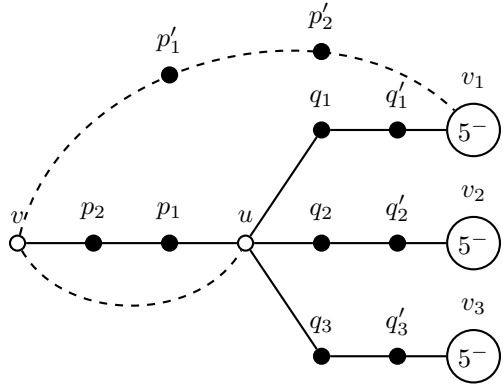
\begin{figure}[H]
\begin{center}
\begin{tikzpicture}{thick}
\begin{scope}[every node/.style={circle,thick,draw,minimum size=1pt,inner sep=2}]
    \node[label={above:$u$}] (0) at (-2,0) {};
    \node[fill,label={above:$q_2$}] (20) at (-1,0) {};
    \node[fill,label={above:$q'_2$}] (21) at (0,0) {};
    \node[label={above:$v_2$}] (22) at (1,0) {$5^-$};
    \node[fill,label={above:$q_1$}] (30) at (-1,1.5) {};
    \node[fill,label={above:$q'_1$}] (31) at (0,1.5) {};
    \node[label={above:$v_1$}] (32) at (1,1.5) {$5^-$};
    \node[fill,label={above:$q_3$}] (40) at (-1,-1.5) {};
    \node[fill,label={above:$q'_3$}] (41) at (0,-1.5) {};
    \node[label={above:$v_3$}] (42) at (1,-1.5) {$5^-$};

    \node[fill,label={above:$p_1$}] (1) at (-3,0) {};
    \node[fill,label={above:$p_2$}] (2) at (-4,0) {};
    \node[label={above:$v$}] (3) at (-5,0) {};
    
    \node[fill,label={above:$p'_1$}] (5) at (-3,2.23) {};
    \node[fill,label={above:$p'_2$}] (6) at (-1,2.54) {};
\end{scope}

\begin{scope}[every edge/.style={draw=black,thick}]
	\path (0) edge (3);
    \path (0) edge (20);
    \path (0) edge (30);
    \path (0) edge (40);
    
    \path (20) edge (22);
    \path (30) edge (32);
    \path (40) edge (42);
      
   \path (3) edge[bend left=60,dashed] (32);
   \path (3) edge[bend right=60,dashed] (0);
\end{scope}
\end{tikzpicture}
\caption{A $7$-vertex incident to at least four $2$-paths, three of which have $5^-$-endvertices.}
\label{sponsor lemma0 figure}
\end{center}
\end{figure}

\begin{lemma} \label{sponsor lemma1}
Consider $u$ a $7$-vertex that is incident to a unique $3$-path $up_1p_2p_3v$ and let $P=\{p_1,p_2,p_3\}$. If $\rho^*_{G-P}(u)\leq \rho^*_{G-P}(v)$, then $u$ has a neighbor that is neither a $(2,2,0)$-vertex nor a $2$-vertex belonging to a $2$-path.
\end{lemma}

\begin{figure}[H]
\begin{center}
\begin{tikzpicture}{thick}
\begin{scope}[every node/.style={circle,thick,draw,minimum size=1pt,inner sep=2}]
    \node[label={above:$u$}] (0) at (-2,0) {};
	\node[fill,label={above:$q_1$}] (13) at (-1,2.5) {};
    \node[fill,label={above:$q'_1$}] (14) at (0,2.5) {};
    \node[label={above:$v_1$}] (15) at (1,2.5) {};    
    
    \node[fill,label={above:$q_k$}] (20) at (-1,1.5) {};
    \node[fill,label={above:$q'_k$}] (21) at (0,1.5) {};
    \node[label={above:$v_k$}] (22) at (1,1.5) {};
	
	\node[fill,label={above:$w_1$}] (35) at (-1,0) {};    
    \node[fill,label={above:$r_1$}] (30) at (0,0.5) {};
    \node[fill,label={above:$r'_1$}] (31) at (1,0.5) {};
    \node[label={above:$w'_1$}] (32) at (2,0.5) {};
    \node[fill,label={above:$s_1$}] (40) at (0,-0.5) {};
    \node[fill,label={above:$s'_1$}] (41) at (1,-0.5) {};
    \node[label={above:$w''_1$}] (42) at (2,-0.5) {};
    
    \node[fill,label={above:$w_l$}] (55) at (-1,-2) {};
    \node[fill,label={above:$r_l$}] (50) at (0,-1.5) {};
    \node[fill,label={above:$r'_l$}] (51) at (1,-1.5) {};
    \node[label={above:$w'_l$}] (52) at (2,-1.5) {};
    \node[fill,label={above:$s_l$}] (60) at (0,-2.5) {};
    \node[fill,label={above:$s'_l$}] (61) at (1,-2.5) {};
    \node[label={above:$w''_l$}] (62) at (2,-2.5) {};

    \node[fill,label={above:$p_1$}] (3) at (-3,0) {};
    \node[fill,label={above:$p_2$}] (4) at (-4,0) {};
    \node[fill,label={above:$p_3$}] (5) at (-5,0) {};
    \node[label={above:$v$}] (6) at (-6,0) {};
    
    \node[draw=none] (8) at (-0.5,2.1) {$\dots$};
    
    \node[draw=none] (9) at (-1,-1) {$\dots$};
\end{scope}

\begin{scope}[every edge/.style={draw=black,thick}]
	\path (0) edge (6);
    \path (0) edge (13);
    \path (0) edge (20);
	\path (0) edge (35);
	\path (0) edge (55);    
    
    \path (35) edge (30);
    \path (35) edge (40);
    \path (55) edge (50);
    \path (55) edge (60);
      \path (13) edge (15);
      \path (20) edge (22);
      \path (30) edge (32);
      \path (40) edge (42);
      \path (50) edge (52);
      \path (60) edge (62);
   
\end{scope}
\end{tikzpicture}
\caption{A $7$-vertex that is incident to a $3$-path, $k$ $2$-paths, and $l$ $(2,2,0)$-vertices where $k+l=6$.}
\label{sponsor lemma1 figure}
\end{center}
\end{figure}
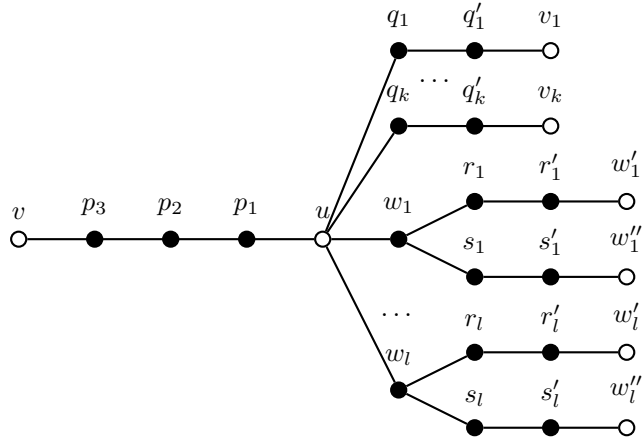

\begin{proof}
Suppose by contradiction that $u$ is incident to $k$ $2$-paths $uq_iq'_iv_i$ for $1\leq i\leq k$ and adjacent to $l$ $(2,2,0)$-vertices $w_j$ for $1\leq j\leq l$ where $k+l=6$. For all $1\leq j\leq l$, let $w_jr_jr'_jw'_j$ and $w_js_js'_jw''_j$ be the $2$-paths incident to $w_j$.
Due to \Cref{3-path lemma}, \Cref{2-path lemma} and \Cref{2-path lemma1}, $u$ is distinct from $v,v_1,\dots,v_k,w_1,\dots,w_l,w'_1,\dots,w'_l,w''_1,\dots,w''_l$ and for all $1\leq j\leq l$, $w_j$ is distinct from $w'_1,\dots,w'_l,w''_1,\dots,w''_l$.

We claim that:
\begin{claim} \label{claim sponsor lemma1}
For all subgraph $H$ of $G-P$, $\rho^*_H(v)\geq 1$.
\end{claim}

\begin{proof}
Indeed, recall that $\rho^*_{G-P}(u)\leq \rho^*_{G-P}(v)$ and by \Cref{pot 3-path lemma}, $\rho^*_{G-P}(v)\geq 1$. As a result, $\rho^*_H(v)\geq \rho^*_{G-P}(v)\geq 1$ by \Cref{potsubgraph}.
\end{proof}

Now we will prove the lemma for each possible value of $0\leq k\leq 6$.
\begin{itemize}
\item Suppose that $k=0$.\\
Let $H=G-(\{u,p_1,p_2\}\cup\{w_j,r_j,s_j|1\leq j\leq 6\}$). By minimality of $G$, there exists a coloring of $H$. We will extend this coloring to $G$:
\begin{itemize}
\item If $x\in V(H)$, then $\phi(x)=\psi(x)$.
\item Note that $d^*_G(p_2)=4$ and $d^*_G(r_j)=d^*_G(s_j)=5$ for all $1\leq j\leq 6$ so we can always color them last.
\item We color $w_1,w_2,\dots,w_6$ since they have six available colors each.
\item We color $p_1$ then $u$.
\end{itemize}

We obtain a valid coloring of $G$ so $k\neq 0$. 
 
\item Suppose that $k=1$.
Let $H=G-(\{u,p_1,p_2,p_3,q_1,q'_1\}\cup\{w_j,r_j,s_j|1\leq j\leq 5\})$. We add the 3-path $vp'_1p'_2p'_3v_1$ in $H$, let $P'=\{p'_1,p'_2,p'_3\}$ and let $H+P'$ be the resulting graph. Since $\rho^*_H(vv_1)\geq \rho^*_H(v)\geq 1$ by \Cref{potsuperset} and \Cref{claim sponsor lemma1}, we get $\mad(H+P')\leq \frac{18}{7}$ by \Cref{potlemma}. By minimality of $G$, there exists a coloring $\psi$ of $H$. We will extend this coloring to $G$:
\begin{itemize}
\item If $x\in V(H)$, then $\phi(x)=\psi(x)$.
\item Let $\phi(p_3)=\psi(p'_1)$ and $\phi(q'_1)=\psi(p'_3)$.
\item Note that $d^*_G(p_2)=4$ and $d^*_G(r_j)=d^*_G(s_j)=5$ for all $1\leq j\leq 5$ so we can always color them last.
\item We color $w_1,w_2,\dots,w_5$ and $q_1$ since they have six available colors each.
\item We can color $u$ and $p_1$ unless they have exactly the same color left which is impossible since they see the same six colors in $\phi(w_1),\phi(w_2),\dots,\phi(w_5),\phi(q_1)$ and $p_1$ sees $\phi(p_3)=\psi(p'_1)\neq\psi(p'_3)=\phi(q'_1)$ which $u$ sees.
\end{itemize}

We obtain a valid coloring of $G$ so $k\neq 1$.

\item Suppose that $k\geq 2$.

Let $H=G-(\{u,p_1,p_2,p_3\}\cup\{q_i,q'_i|1\leq i\leq k\})$. Note that $\mad(H)\leq \frac{18}{7}$ since $H$ is a subgraph of $G$.

First, observe that by minimality of $G$, there exists a coloring $\psi$ of $H$. If we can define a coloring $\phi$ that extends $\psi$ to $G$, then we obtain a contradiction. So, let us see the potential problems.
\begin{itemize}
\item First, if $x\in V(H)\setminus\{w_j,r_j,s_j|1\leq j\leq l\}$, then we are forced to repeat the same colors for $x$. Thus, let $\phi(x)=\psi(x)$.
\item For all $1\leq i\leq k$, we might have only one choice of colors for $q'_i$ so we color them accordingly. The same holds for $p_3$.
\item Since $d^*_G(p_2)=4$ and $d^*_G(r_j)=d^*_G(s_j)=5$ for all $1\leq j\leq l$, we can always color them last.
\item The remaining uncolored vertices are exactly $N_G(u)\cup\{u\}$. Let $L(x)$ be the list of available colors left for a vertex $x$. Observe that $|L(u)|\geq 8-k\geq 2$, $|L(q_1)|,\dots,|L(q_k)|,|L(w_1)|,\dots,|L(w_k)|\geq 6$, and $|L(p_1)|\geq 7$.
\item Due to \Cref{Hall}, the only two reasons that makes these eight remaining vertices uncolorable are the following:
\begin{itemize}
\item We have seven vertices in $N_G(u)\cup\{u\}\setminus\{p_1\}$ but $|L(u)\cup L(q_1)\cup\dots\cup L(q_k)\cup L(w_1)\cup\dots\cup L(w_l)|\leq 6$. Since $|L(q_1)|,\dots,|L(q_k)|,|L(w_1)|,\dots,|L(w_k)|\geq 6$, we have $L(q_1)=\dots=L(q_k)=L(w_1)=\dots=L(w_l)$ and $|L(q_1)|=6$. In other words, $q_1,\dots,q_k,w_1,\dots,w_k$ all see the same two colors. More precisely, $\{\phi(q'_1),\phi(v_1)\}=\dots=\{\phi(q'_k),\phi(v_k)\}=\{\phi(r'_1),\phi(s'_1)\}=\dots=\{\phi(r'_l),\phi(s'_l)\}$.
\item Or, we have eight vertices in $N_G(u)\cup\{u\}$ but $|L(u)\cup L(p_1)\cup L(q_1)\cup\dots\cup L(q_k)\cup L(w_1)\cup\dots\cup L(w_l)|\leq 7$. Since $|L(p_1)|\geq 7$, we have $|L(p_1)|=7$. Moreover, $L(q_1),\dots,L(q_k),L(w_1),\dots,L(w_k)\subseteq L(p_1)$. In other words, $q_1,\dots,q_k,w_1,\dots,w_k$ all see $\phi(p_3)$. More precisely, $\phi(p_3)\in \{\phi(q'_i),\phi(v_i)\}$ for all $1\leq i\leq k$ and $\phi(p_3)\in \{\phi(r'_j),\phi(s'_j)\}$ for all $1\leq j\leq l$. 
\end{itemize}
\end{itemize}   
To solve these two problems, the idea is to add two paths (or edges) to $H$, each one preventing one problem. If we can add these two paths, we can define a valid coloring $\phi$ of $G$, thus obtaining a contradiction. As a consequence, we cannot add both paths. However, it results in an upper bound on the potential of the endvertices of the added paths. In \Cref{claim2 sponsor lemma1,claim3 sponsor lemma1,claim4 sponsor lemma1,claim5 sponsor lemma1}, we show these upper bounds by using this technique of adding two paths to the graph $H$ and constructing a valid coloring of $G$. Once we obtain all of these inequalities on the potential in $H$ of $v,v_1,\dots,v_k,w_1,\dots,w_l$, we show, for each value of $k$, that the obtained set of inequalities is not feasible, thus obtaining a contradiction.

\begin{claim}\label{claim2 sponsor lemma1}
For $k\geq 2$ and $j\geq 1$, if there exists $1\leq i\neq i'\leq k$ and $1\leq j\leq l$ such that $\rho^*_H(v_iv_{i'})\geq 3$ and $\rho^*_H(vw_j)\geq 7$, then $\rho^*_H(vw_jv_iv_{i'})\leq 9$.
\end{claim}

\begin{proof}
Suppose by contradiction that, w.l.o.g., $\rho^*_H(v_1v_2)\geq 3$, $\rho^*_H(vw_1)\geq 7$, and $\rho^*_H(vw_1v_1v_2)\geq 10$. We add the 2-path $v_1p'_1p'_2v_2$ in $H$, let $P'=\{p'_1,p'_2\}$, and let $H+P'$ be the resulting graph. Since $\rho^*_H(v_1v_2)\geq 3$, we get $\mad(H+P')\leq \frac{18}{7}$ by \Cref{potlemma}. Observe that $\rho^*_{H+P'}(vw_1)\geq 7$, otherwise, by \Cref{potlemma2}, we get  $6\geq \rho^*_{H+P'}(vw_1)=\rho^*_H(vw_1)\geq 7$ or $10\leq \rho^*_H(vw_1v_1v_2)\leq 6+3=9$ which are both contradictions. Now, we add the edge $e=vw_1$ in $H+P'$ and by \Cref{potlemma}, we have $\mad(H+P'+e)\leq \frac{18}{7}$. By minimality of $G$, there exists a coloring $\psi$ of $H+P'+e$. We define $\phi$ a coloring of $G$ as follows:

\begin{itemize}
\item If $x\in V(H)\setminus\{w_j,r_j,s_j|1\leq j\leq l\}$, then $\phi(x)=\psi(x)$.
\item We color $q'_i$ for all $3\leq i\leq k$ since they all have at least one available color each.
\item Let $\phi(p_3) = \phi(w_1) = \psi(w_1)$, $\phi(q'_1) = \psi(p'_1)$ and $\phi(q'_2)=\psi(p'_2)$.
\item Note that $d^*_G(p_2)=4$ and $d^*_G(r_j)=d^*_G(s_j)=5$ for all $1\leq j\leq l$ so we can always color them last.
\item We color $u$ who has at least two available colors left as $u$ sees $\phi(q'_1),\dots,\phi(q'_k)$, and $\phi(w_1)$.
\item Then, we color $w_2,\dots,w_l$ and $q_3,\dots,q_k$ since there are three of them and each one has at least three available colors left.
\item Now, we color $q_1$ and $q_2$ which each has at least one color left. These colors are different since $q_1$ and $q_2$ see the same five colors in $\phi(u)$, $\phi(w_1),\dots,\phi(w_l)$ and $\phi(q_3),\dots,\phi(q_k)$ and $q_1$ sees $\{\phi(q'_1),\phi(v_1)\}=\{\psi(v_1),\psi(p'_1)\}\neq\{\psi(p'_2),\psi(v_2)\}=\{\phi(q'_2),\phi(v_2)\}$ which $q_2$ sees.
\item We color $p_1$ since it sees eight colored vertices but two of them, namely $w_1$ and $p_3$ have the same color.
\end{itemize}

We obtain a valid coloring of $G$ which is a contradiction.
\end{proof}

\begin{claim}\label{claim3 sponsor lemma1}
For $k\geq 1$ and $j\geq 2$, if there exist $1\leq i\leq k$ and $1\leq  j\neq j'\leq l$ such that $\rho^*_H(vw_{j'})\geq 7$ and $\rho^*_H(v_iw_j)\geq 7$, then $\rho^*_H(vw_{j'}v_iw_j)\leq 13$.
\end{claim}

\begin{proof}

Suppose by contradiction that, w.l.o.g., $\rho^*_H(vw_2)\geq 7$, $\rho^*_H(v_1w_1)\geq 7$, and $\rho^*_H(vw_2v_1w_1)\geq 14$. We add the edge $e=vw_2$ in $H$ and let $H+e$ be the resulting graph. Since $\rho^*_H(vw_2)\geq 7$, we get $\mad(H+e)\leq \frac{18}{7}$ by \Cref{potlemma}. We have $\rho^*_{H+e}(v_1w_1)\geq 7$, otherwise, by \Cref{potlemma2}, we get  $6\geq \rho^*_{H+e}(v_1w_1)=\rho^*_H(v_1w_1)\geq 7$ or $14\leq \rho^*_H(vw_2v_1w_1)\leq 6+7=13$ which are both contradictions. Now, we add the edge $e'=v_1w_1$ in $H+e$. So, by \Cref{potlemma}, we have $\mad(H+e+e')\leq \frac{18}{7}$. By minimality of $G$, there exists a coloring $\psi$ of $H+e+e'$. We define $\phi$ a coloring of $G$ as follows:

\begin{itemize}
\item If $x\in V(H)\setminus\{w_j,r_j,s_j|1\leq j\leq l\}$, then $\phi(x)=\psi(x)$.
\item We color $q'_i$ for all $2\leq i\leq k$ since they all have at least one available color each.
\item Let $\phi(p_3) = \phi(w_2) = \psi(w_2)$ and $\phi(w_1) = \phi(q'_1) = \psi(q'_1)$.
\item Note that $d^*_G(p_2)=4$ and $d^*_G(r_j)=d^*_G(s_j)=5$ for all $1\leq j\leq l$ so we can always color them last.
\item We color $u$ who has at least two available colors.
\item Then, we color $w_3,\dots,w_l$ and $q_2,\dots,q_k$ since there are three of them and each one has three available colors left.
\item Now, we color $q_1$ which sees eight colored vertices but two of them, namely $w_1$ and $q'_1$ have the same color.
\item Similarly, we can color $p_1$ since it sees eight colored vertices but two of them, namely $w_2$ and $p_3$ have the same color.
\end{itemize}

We obtain a valid coloring of $G$ which is a contradiction.
\end{proof}

\begin{claim}\label{claim4 sponsor lemma1}
For $k\geq 3$, if there exist three distinct integers $1\leq i,i',i''\leq k$, $\rho^*_H(vv_i'')\geq 3$ and $\rho^*_H(v_iv_{i'})\geq 3$, then $\rho^*_H(vv_iv_{i'}v_{i''})\leq 5$.
\end{claim}

\begin{proof}

Suppose by contradiction that, w.l.o.g., $i=1,i'=2,i''=3$. In other words, $\rho^*_H(vv_3)\geq 3$, $\rho^*_H(v_1v_2)\geq 3$, and $\rho^*_H(vv_1v_2v_3)\geq 6$. We add the 2-path $v_1p'_1p'_2v_2$ in $H$ and let $P'=\{p'_1,p'_2\}$. Since $\rho^*_H(v_1v_2)\geq 3$, we get $\mad(H+P')\leq \frac{18}{7}$ by \Cref{potlemma}. We have $\rho^*_{H+P'}(vv_3)\geq 3$, otherwise, by \Cref{potlemma2}, we get  $2\geq \rho^*_{H+P'}(vv_3)=\rho^*_H(vv_3)\geq 3$ or $6\leq \rho^*_H(v_1v_2vv_3)\leq 2+3=5$ which are both contradictions. Now, we add the 2-path $vp''_1p''_2v_3$ in $H+P'$ and let $P''=\{p''_1,p''_2\}$. So, by \Cref{potlemma}, we have $\mad(H+P'+P'')\leq \frac{18}{7}$. By minimality of $G$, there exists a coloring $\psi$ of $H+P'+P''$. We define $\phi$ a coloring of $G$ as follows:

\begin{itemize}
\item If $x\in V(H)\setminus\{w_j,r_j,s_j|1\leq j\leq l\}$, then $\phi(x)=\psi(x)$.
\item We color $q'_i$ for all $4\leq i\leq k$ since they all have at least one available color each.
\item Let $\phi(p_3) = \psi(p''_1)$, $\phi(q'_3) = \psi(p''_2)$, $\phi(q'_1)=\psi(p'_1)$, and $\phi(q'_2)=\psi(p'_2)$.
\item Note that $d^*_G(p_2)=4$ and $d^*_G(r_j)=d^*_G(s_j)=5$ for all $1\leq j\leq l$ so we can always color them last.
\item Let $L(x)$ be the list of available colors left for a vertex $x$. Observe that we have $|L(u)|\geq 2$, $|L(w_1)|,\dots,|L(w_l)|,|L(q_1)|,\dots,|L(q_k)|\geq 6$ and $|L(p_1)|\geq 7$. By \Cref{Hall}, the only two ways these eight vertices are not colorable is the following:
\begin{itemize}
\item We have seven vertices $u,w_1,\dots,w_l,q_1,\dots,q_l$ but $|L(u)\cup L(w_1)\cup\dots\cup L(w_l)\cup L(q_1)\cup\dots\cup L(q_l)|\leq 6$. However, this is not possible since $q_1$ sees $\{\phi(q'_1),\phi(v_1)\}=\{\psi(v_1),\psi(p'_1)\}\neq\{\psi(p'_2),\psi(v_2)\}=\{\phi(q'_2),\phi(v_2)\}$ which $q_2$ sees. So, $|L(q_1)\cup L(q_2)|\geq 7$.
\item We have $|L(u)\cup L(p_1)\cup L(w_1)\cup\dots\cup L(w_l)\cup L(q_1)\cup\dots\cup L(q_l)|\leq 7$. However, this is not possible since $p_1$ sees $\phi(p_3)=\psi(p''_1)\notin\{\psi(p''_2),\psi(v_3)\}=\{\phi(q'_3),\phi(v_3)\}$ which $q_3$ sees. So, $|L(p_1)\cup L(q_3)|\geq 8$.
\end{itemize}
\end{itemize}

We obtain a valid coloring of $G$ which is a contradiction.
\end{proof}

\begin{claim}\label{claim5 sponsor lemma1}
For $k\geq 2$ and $j\geq 1$, if there exists $1\leq i \neq i'\leq k$ and $1\leq j\leq l$ such that $\rho^*_H(vv_{i'})\geq 3$ and $\rho^*_H(v_iw_j)\geq 7$, then $\rho^*_H(vv_{i'}v_iw_j)\leq 9$.
\end{claim}

\begin{proof}
Suppose by contradiction that, w.l.o.g., $\rho^*_H(vv_2)\geq 3$, $\rho^*_H(v_1w_1)\geq 7$, and $\rho^*_H(vv_2v_1w_1)\geq 10$. We add the 2-path $vp'_1p'_2v_2$ in $H$ and let $P'=\{p'_1,p'_2\}$. Since $\rho^*_H(vv_2)\geq 3$, we get $\mad(H+P')\leq \frac{18}{7}$ by \Cref{potlemma}. We have $\rho^*_{H+P'}(v_1w_1)\geq 7$, otherwise, by \Cref{potlemma2}, we get  $6\geq \rho^*_{H+P'}(vw_1)=\rho^*_H(v_1w_1)\geq 7$ or $10\leq \rho^*_H(vv_2v_1w_1)\leq 6+3=9$ which are both contradictions. Now, we add the edge $e=v_1w_1$ in $H+P'$. So, by \Cref{potlemma}, we have $\mad(H+P'+e)\leq \frac{18}{7}$. By minimality of $G$, there exists a coloring $\psi$ of $H+P'+e$. We define $\phi$ a coloring of $G$ as follows:

\begin{itemize}
\item If $x\in V(H)\setminus\{w_j,r_j,s_j|1\leq j\leq l\}$, then $\phi(x)=\psi(x)$.
\item We color $q'_i$ for all $3\leq i\leq k$ since they all have at least one available color each.
\item Let $\phi(p_3) = \psi(p'_1)$, $\phi(q'_2) = \psi(p'_2)$ and $\phi(q'_1) = \phi(w_1) = \psi(w_1)$.
\item Note that $d^*_G(p_2)=4$ and $d^*_G(r_j)=d^*_G(s_j)=5$ for all $1\leq j\leq l$ so we can always color them last.
\item Let $L(x)$ be the list of available colors left for a vertex $x$. Observe that we have $|L(u)|\geq 2$, $|L(w_2)|,\dots,|L(w_l)|,|L(q_1)|,\dots,|L(q_k)|\geq 5$, and $|L(p_1)|\geq 6$. By \Cref{Hall}, the only two ways these seven vertices are not colorable is the following:
\begin{itemize}
\item We have six vertices $u,w_2,\dots,w_l,q_1,\dots,q_l$ but $|L(u)\cup L(w_1)\cup\dots\cup L(w_l)\cup L(q_1)\cup\dots\cup L(q_l)|\leq 5$. However, this is not possible since $q_1$ sees three colored vertices in $v_1$, $q'_1$ and $w_1$ but $\phi(q'_1)=\phi(w_1)$ so, $|L(q_1)|\geq 6$.
\item We have $|L(u)\cup L(p_1)\cup L(w_1)\cup\dots\cup L(w_l)\cup L(q_1)\cup\dots\cup L(q_l)|\leq 6$. However, this is not possible since $p_1$ sees $\phi(p_3)=\psi(p'_1)\notin\{\psi(p'_2),\psi(v_2)\}=\{\phi(q'_2),\phi(v_2)\}$ which $q_2$ sees. So, $|L(p_1)\cup L(q_2)|\geq 7$.
\end{itemize}
\end{itemize}

We obtain a valid coloring of $G$ which is a contradiction.
\end{proof}

Given \Cref{claim2 sponsor lemma1,claim3 sponsor lemma1,claim4 sponsor lemma1,claim5 sponsor lemma1}, we can show upper bounds on the potential on some subsets of vertices of $G$. However, due to \Cref{potlowerbound}, the lower bounds on the potential of these subsets exceed the upper bounds, which is a contradiction.

First, recall that $H=G-(\{u,p_1,p_2,p_3\}\cup\{q_i,q'_i|1\leq i\leq k\})$ and observe that:
\begin{observation}\label{obs sponsor lemma1}
For $0\leq i\leq k$ and $0\leq j\leq l$, by applying \Cref{potlowerbound} to the graph $G-\{up_1,uq_{i+1},\dots,uq_{k}$, $uw_{j+1},\dots,uw_l\}$ with $A=\{u\}\cup\{q_x,q'_x|1\leq x\leq i\}$ and $S=\{v_1,\dots,v_i,w_1,\dots,w_j\}$,  we have $\rho^*_H(S)\geq 3i+7j-9$. Similarly, by applying \Cref{potlowerbound} to the graph $G-\{uq_{i+1},\dots,uq_{k},uw_{j+1},\dots,uw_l\}$ with $A=\{u,p_1,p_2,p_3\}\cup\{q_x,q'_x|1\leq x\leq i\}$ and $S=\{v,v_1,\dots,v_i,w_1,\dots,w_j\}$, $\rho^*_H(S)\geq 3i+7j-8$.
\end{observation}

Indeed, by \Cref{potlowerbound}, we obtain the first inequality through the following calculations: 
\begin{align*}
\rho^*_H(S) & \geq 7(i+j) - (9(2i + 1) - 7\cdot 2i)\\
& \geq 7i + 7j - (18i + 9 - 14i)\\
& \geq 7i + 7j - 18i - 9 + 14i\\
& \geq 3i + 7j - 9
\end{align*}

Similarly, we can obtain the second equation through the same kind of calculations.

Now, we consider the different values of $2\leq k\leq 6$.

\textbf{For $\mathbf{k=2}$:}
\begin{itemize}
\item Suppose that $\rho^*_H(vw_j)\leq 6$ for all $1\leq j\leq 4$.\\
As a result, for all $1\leq i\leq 2$, $\rho^*_H(vw_1)+\dots+\rho^*_H(vw_4)+\rho^*_H(vv_i)\leq 6\cdot 4 + \rho^*_H(vv_i) = 24 + \rho^*_H(vv_i)$. We also have $\rho^*_H(vw_1)+\dots+\rho^*_H(vw_4)+\rho^*_H(vv_i)\geq \rho^*_H(vv_iw_1\dots w_4) + 4\rho^*_H(v)\geq 3\cdot 1 +7\cdot 4 - 8 + 4\rho^*_H(v) \geq 23 + 4\cdot 1 = 27$ where the first inequality corresponds to \Cref{potadd}, the second one to \Cref{obs sponsor lemma1}, and the third one to \Cref{claim sponsor lemma1} (from now on, we will repeat the same scheme). So, $\rho^*_H(vv_i)\geq 27-24 = 3$ for all $1\leq i\leq 2$.

Suppose there exist $1\leq i\leq 2$ and $1\leq j \leq 4$ such that $\rho^*_H(v_iw_j)\geq 7$. Say w.l.o.g. that $\rho^*_H(v_2w_2)\geq 7$, then by \Cref{claim5 sponsor lemma1}, $\rho^*_H(vv_1v_2w_2)\leq 9$. Thus, we get $\rho^*_H(vv_1v_2w_2)+\rho^*_H(vw_1)+\rho^*_H(vw_3)+\rho^*_H(vw_4)\leq 9 + 3\cdot 6 = 27$. However, by \Cref{potadd}, \Cref{obs sponsor lemma1} then \Cref{claim sponsor lemma1}, we have $\rho^*_H(vv_1v_2w_2)+\rho^*_H(vw_1)+\rho^*_H(vw_3)+\rho^*_H(vw_4)\geq \rho^*_H(vv_1v_2w_1\dots w_4) + 3\rho^*_H(v)\geq 3\cdot 2 + 7\cdot 4 - 8 + 3\rho^*_H(v) \geq 26 + 3 = 29$.

So, for all $1\leq i\leq 2$ and $1\leq j \leq 4$, $\rho^*_H(v_iw_j)\leq 6$. Thus, we get $\rho^*_H(v_1w_1)+\rho^*_H(v_2w_2)+\rho^*_H(vw_3)+\rho^*_H(vw_4)\leq 4\cdot 6 = 24$. However, by \Cref{potadd}, \Cref{obs sponsor lemma1} then \Cref{claim sponsor lemma1}, we have $\rho^*_H(v_1w_1)+\rho^*_H(v_2w_2)+\rho^*_H(vw_3)+\rho^*_H(vw_4)\geq \rho^*_H(vv_1v_2w_1\dots w_4) + \rho^*_H(v)\geq 3\cdot 2 + 7\cdot 4 - 8 +\rho^*_H(v) \geq 26 + 1 = 27$.

\item Suppose w.l.o.g. that $\rho^*_H(vw_1)\geq 7$.
\begin{itemize}
\item Suppose that for all $1\leq i\leq 2$ and $2\leq j\leq 4$, we have $\rho^*_H(v_iw_j)\leq 6$.\\
As a result, $\rho^*_H(v_2w_2)+\rho^*_H(v_1w_3)+\rho^*_H(v_1w_4)+\rho^*_H(v_1w_1)\leq 3\cdot 6 + \rho^*_H(v_1w_1) = 18 + \rho^*_H(v_1w_1)$. Moreover, by \Cref{potadd} then \Cref{obs sponsor lemma1}, $\rho^*_H(v_2w_2)+\rho^*_H(v_1w_3)+\rho^*_H(v_1w_4)+\rho^*_H(v_1w_1)\geq \rho^*_H(v_1v_2w_1\dots w_4) + 2\rho^*_H(v_1)\geq 3\cdot 2 + 7\cdot 4 - 9 = 25$. So, we get $\rho^*_H(v_1w_1) \geq 25-18 = 7$.

Suppose there exists $2\leq j\leq 4$ such that $\rho^*_H(vw_j)\geq 7$. Say w.l.o.g. that $\rho^*_H(vw_2)\geq 7$, then by \Cref{claim3 sponsor lemma1}, $\rho^*_H(v_1w_1vw_2)\leq 13$. We get $\rho^*_H(v_1w_1vw_2) + \rho^*_H(v_2w_3)+\rho^*_H(v_2w_4)\leq 13 + 2\cdot 6 = 25$. However, by \Cref{potadd} then \Cref{obs sponsor lemma1}, $\rho^*_H(v_1w_1vw_2) + \rho^*_H(v_2w_3)+\rho^*_H(v_2w_4)\geq \rho^*_H(vv_1v_2w_1\dots w_4) + \rho^*_H(v_2)\geq 3\cdot 2 + 7\cdot 4 - 8 = 26$.

So, $\rho^*_H(vw_j)\leq 6$ for all $2\leq j\leq 4$. We get $\rho^*_H(vw_2) + \rho^*_H(v_1w_3)+\rho^*_H(v_2w_4)\leq 3\cdot 6 = 18$. However, by \Cref{potadd} then \Cref{obs sponsor lemma1}, $\rho^*_H(vw_2) + \rho^*_H(v_1w_3)+\rho^*_H(v_2w_4)\geq \rho^*_H(vv_1v_2w_2w_3w_4)\geq 3\cdot 2 + 7\cdot 3 - 8 = 19$.

\item Suppose that there exist $1\leq i\leq 2$ and $2\leq j\leq 4$ such that $\rho^*_H(v_iw_j)\geq 7$. Say w.l.o.g. $\rho^*_H(v_2w_2)\geq 7$.\\
By \Cref{claim3 sponsor lemma1}, $\rho^*_H(vw_1v_2w_2)\leq 13$. As a result, by \Cref{potadd} then \Cref{obs sponsor lemma1}, we get $\rho^*_H(vw_1v_2w_2) + \rho^*_H(v_1w_3)+\rho^*_H(v_1w_4)\geq \rho^*_H(vv_1v_2w_1\dots w_4) + \rho^*_H(v_1) \geq 3\cdot 2 + 7\cdot 4 - 8 = 26$. So, $\rho^*_H(v_1w_3)+\rho^*_H(v_1w_4)\geq 26-13=13$ and $\rho^*_H(v_1w_3)\geq 7$ w.l.o.g.

By \Cref{claim3 sponsor lemma1}, $\rho^*_H(vw_1v_1w_3)\leq 13$. As a result, by \Cref{potadd} then \Cref{obs sponsor lemma1}, we get $\rho^*_H(vw_1v_2w_2) + \rho^*_H(vw_1v_1w_3)+\rho^*_H(v_1w_4)\geq \rho^*_H(vv_1v_2w_1\dots w_4) + \rho^*_H(vw_1)+\rho^*_H(v_1)\geq 3\cdot 2 + 7\cdot 4 - 8 + 7 = 33$. So, $\rho^*_H(v_1w_4)\geq 33-2\cdot 13 = 7$.

By \Cref{claim3 sponsor lemma1}, $\rho^*_H(vw_1v_1w_4)\leq 13$. Finally, we have $\rho^*_H(vw_1v_2w_2) + \rho^*_H(vw_1v_1w_3)+\rho^*_H(vw_1v_1w_4)\leq 3\cdot 13 = 39$. However, by \Cref{potadd} then \Cref{obs sponsor lemma1}, $\rho^*_H(vw_1v_2w_2) + \rho^*_H(vw_1v_1w_3)+\rho^*_H(vw_1v_1w_4)\geq \rho^*_H(vv_1v_2w_1\dots w_4) + 2\rho^*_H(vw_1) +\rho^*_H(v_1) \geq 3\cdot 2 + 7\cdot 4 - 8 + 2\cdot 7 = 40$.
\end{itemize}

\end{itemize}

\textbf{For $\mathbf{k=3}$:}\\
Suppose that $\rho^*_H(vw_j)\leq 6$ for all $1\leq j\leq 3$. 

Let $\{i,i',i''\}$ be any permutation of $\{1,2,3\}$, $\rho^*_H(vw_1)+\rho^*_H(vw_2)+\rho^*_H(vw_3)+\rho^*_H(v_iv_{i'})\leq 3\cdot 6 + \rho^*_H(v_iv_{i'}) = 18 + \rho^*_H(v_iv_{i'})$. Moreover, by \Cref{potadd}, \Cref{obs sponsor lemma1} then \Cref{claim sponsor lemma1}, $\rho^*_H(vw_1)+\rho^*_H(vw_2)+\rho^*_H(vw_3)+\rho^*_H(v_iv_{i'}) \geq \rho^*_H(vv_iv_{i'}w_1w_2w_3) + 2\rho^*_H(v)\geq 3\cdot 2 + 7\cdot 3 - 8 + 2 = 21$. So, we get $\rho^*_H(v_iv_{i'})\geq 21-18=3$.

If $\rho^*_H(vv_{i''})\geq 3$, then by \Cref{claim4 sponsor lemma1}, $\rho^*_H(vv_iv_{i'}v_{i''})\leq 5$. Since $\rho^*_H(vv_iv_{i'}v_{i''}) = \rho^*_H(vv_1v_2v_3)$, we get $\rho^*_H(vw_1)+\rho^*_H(vw_2)+\rho^*_H(vw_3)+\rho^*_H(vv_1v_2v_3)\leq 3\cdot 6 + 5 = 23$. However, by \Cref{potadd}, \Cref{obs sponsor lemma1} then \Cref{claim sponsor lemma1}, $\rho^*_H(vw_1)+\rho^*_H(vw_2)+\rho^*_H(vw_3)+\rho^*_H(vv_1v_2v_3) \geq \rho^*_H(vv_1v_2v_3w_1w_2w_3) + 3\rho^*_H(v) \geq 3\cdot 3 + 7\cdot 3 - 8 + 3 = 25$.

Observe the previous argument holds for any permutation of $\{i,i',i''\}=\{1,2,3\}$. So $\rho^*_H(vv_{i''})\leq 2$ for all $1\leq i''\leq 3$. Thus, we get $\rho^*_H(vw_1)+\rho^*_H(vw_2)+\rho^*_H(vw_3)+\rho^*_H(vv_1)+\rho^*_H(vv_2)+\rho^*_H(vv_3)\leq 3\cdot 6 + 3\cdot 2 = 24$. However, by \Cref{potadd}, \Cref{obs sponsor lemma1} then \Cref{claim sponsor lemma1}, $\rho^*_H(vw_1)+\rho^*_H(vw_2)+\rho^*_H(vw_3)+\rho^*_H(vv_1)+\rho^*_H(vv_2)+\rho^*_H(vv_3) \geq \rho^*_H(vv_1v_2v_3w_1w_2w_3) + 5\rho^*_H(v) \geq 3\cdot 3 + 7\cdot 3 - 8 + 5 = 27$.

Now, suppose there exists $1\leq j\leq 3$ such that $\rho^*_H(vw_j)\geq 7$. Say w.l.o.g. that $\rho^*_H(vw_1)\geq 7$
\begin{itemize}
\item Suppose that there exist $1\leq i\leq 3$ and $2\leq j\leq 3$ such that $\rho^*_H(v_iw_j)\geq 7$. Say w.l.o.g. that $\rho^*_H(v_3w_3)\geq 7$. By \Cref{claim3 sponsor lemma1}, $\rho^*_H(vw_1v_3w_3)\leq 13$.
\begin{itemize}
\item If $\rho^*_H(v_1v_2)\leq 2$, then $\rho^*_H(vw_1v_3w_3) + \rho^*_H(v_1v_2) + \rho^*_H(v_2w_2)\leq 13 + 2 + \rho^*_H(v_2w_2) = 15 + \rho^*_H(v_2w_2)$. By \Cref{potadd} then \Cref{obs sponsor lemma1}, $\rho^*_H(vw_1v_3w_3) + \rho^*_H(v_1v_2) + \rho^*_H(v_2w_2)\geq \rho^*_H(vv_1v_2v_3w_1w_2w_3) + \rho^*_H(v_2) \geq 3\cdot 3 + 7\cdot 3 - 8 = 22$. So, $\rho^*_H(v_2w_2)\geq 22-15=7$. 

By \Cref{claim3 sponsor lemma1}, $\rho^*_H(vw_1v_2w_2)\leq 13$. Thus, we get $\rho^*_H(vw_1v_2w_2)+\rho^*_H(vw_1v_3w_3)+\rho^*_H(v_1v_2)\leq 2\cdot 13 + 2 = 28$. However, by \Cref{potadd} then \Cref{obs sponsor lemma1}, $\rho^*_H(vw_1v_2w_2)+\rho^*_H(vw_1v_3w_3)+\rho^*_H(v_1v_2)\geq \rho^*_H(vv_1v_2v_3w_1w_2w_3) + \rho^*_H(vw_1) + \rho^*_H(v_2)\geq 3\cdot 3 + 7\cdot 3 - 8 + 7 = 29$.

\item If $\rho^*_H(v_1v_2)\geq 3$, then by \Cref{claim2 sponsor lemma1}, $\rho^*_H(vw_1v_1v_2)\leq 9$. So, we get $\rho^*_H(vw_1v_1v_2)+\rho^*_H(vw_1v_3w_3)+\rho^*_H(v_2w_2)\leq 9+13+\rho^*_H(v_2w_2) = 22 + \rho^*_H(v_2w_2)$. By \Cref{potadd} then \Cref{obs sponsor lemma1}, $\rho^*_H(vw_1v_1v_2)+\rho^*_H(vw_1v_3w_3)+\rho^*_H(v_2w_2)\geq \rho^*_H(vv_1v_2v_3w_1w_2w_3) + \rho^*_H(vw_1) + \rho^*_H(v_2)\geq 3\cdot 3 + 7\cdot 3 - 8 + 7 = 29$. As a result, $\rho^*_H(v_2w_2)\geq 29-22=7$.

By \Cref{claim3 sponsor lemma1}, $\rho^*_H(vw_1v_2w_2)\leq 13$. Thus, $\rho^*_H(vw_1v_3w_3) + \rho^*_H(vw_1v_2w_2) + \rho^*_H(vw_1v_1v_2) \leq 2\cdot 13 + 9 = 35$. However, by \Cref{potadd} then \Cref{obs sponsor lemma1}, $\rho^*_H(vw_1v_3w_3) + \rho^*_H(vw_1v_2w_2) + \rho^*_H(vw_1v_1v_2) \geq \rho^*_H(vv_1v_2v_3w_1w_2w_3) + 2\rho^*_H(vw_1) + \rho^*_H(v_2)\geq 3\cdot 3 + 7\cdot 3 - 8 + 2\cdot 7 = 36$. 

\end{itemize}

\item Suppose that $\rho^*_H(v_iw_j)\leq 6$ for all $1\leq i\leq 3$ and $2\leq j\leq 3$.

If $\rho^*_H(v_1v_2)\geq 3$, then by \Cref{claim2 sponsor lemma1}, $\rho^*_H(vw_1v_1v_2)\leq 9$. Thus, $\rho^*_H(vw_1v_1v_2) + \rho^*_H(v_2w_2) + \rho^*_H(v_3w_3)\leq 9 + 2\cdot 6 = 21$. However, by \Cref{potadd} then \Cref{obs sponsor lemma1}, $\rho^*_H(vw_1v_1v_2) + \rho^*_H(v_2w_2) + \rho^*_H(v_3w_3)\geq \rho^*_H(vv_1v_2v_3w_1w_2w_3) + \rho^*_H(v_2) \geq 3\cdot 3 + 7\cdot 3 - 8 =22$.

So, $\rho^*_H(v_1v_2)\leq 2$. Thus, $\rho^*_H(v_1v_2) + \rho^*_H(v_2w_2) + \rho^*_H(v_3w_3)\leq 2 + 2\cdot 6 = 14$. Moreover, by \Cref{potadd} then \Cref{obs sponsor lemma1}, $\rho^*_H(v_1v_2) + \rho^*_H(v_2w_2) + \rho^*_H(v_3w_3)\geq \rho^*_H(v_1v_2v_3w_2w_3) + \rho^*_H(v_2) \geq 3\cdot 3 + 7\cdot 2 - 9 + \rho^*_H(v_2) = 14 + \rho^*_H(v_2)$. As a result, $\rho^*_H(v_2) \leq 14 - 14 = 0$. However, $\rho^*_H(v_1w_2)+\rho^*_H(v_3w_3) + \rho^*_H(v_2) \leq 2\cdot 6 + 0 = 12$ and by \Cref{potadd} then \Cref{obs sponsor lemma1}, $\rho^*_H(v_1w_2)+\rho^*_H(v_3w_3) + \rho^*_H(v_2) \geq \rho^*_H(v_1v_2v_3w_2w_3) \geq 3\cdot 3 + 7\cdot 2 - 9 = 14$.   

\end{itemize}    

\textbf{For $\mathbf{k=4}$:}\\
If $\rho^*_H(v_1v_2)\leq 2$ and $\rho^*_H(v_3v_4)\leq 2$ then $\rho^*_H(v_1v_2)+\rho^*_H(v_3v_4)+\rho^*_H(w_1)+\rho^*_H(w_2)\leq 4+\rho^*_H(w_1)+\rho^*_H(w_2)$. Moreover, by \Cref{potadd} then \Cref{obs sponsor lemma1}, $\rho^*_H(v_1v_2)+\rho^*_H(v_3v_4)+\rho^*_H(w_1)+\rho^*_H(w_2)\geq \rho^*_H(v_1v_2v_3v_4w_1w_2)\geq 3\cdot 4 + 7\cdot 2 - 9 = 17$. As a result, $\rho^*_H(w_1)+\rho^*_H(w_2)\geq 17-4=13$. So, $\rho^*_H(w_1)\geq 7$ w.l.o.g. and by \Cref{potsuperset}, $\rho^*_H(w_1v_i)\geq 7$ for all $1\leq i\leq 4$. 

At the same time, $\rho^*_H(v_1v_2)+\rho^*_H(v_3v_4)+\rho^*_H(vw_2)\leq 4+ \rho^*_H(vw_2)$. Moreover, by \Cref{potadd} then \Cref{obs sponsor lemma1}, $\rho^*_H(v_1v_2)+\rho^*_H(v_3v_4)+\rho^*_H(vw_2)\geq \rho^*_H(vv_1v_2v_3v_4w_2)\geq 3\cdot 4 + 7 - 8 = 11$. As a result, $\rho^*_H(vw_2)\geq 11-4=7$.

By \Cref{claim3 sponsor lemma1}, we get $\rho^*_H(vw_2w_1v_i)\leq 13$ for all $1\leq i\leq 4$. As a result, we have $\rho^*_H(v_1v_2)+\rho^*_H(v_3v_4)+\rho^*_H(vw_1w_2v_1)\leq 2\cdot 2 + 13 = 17$. However, by \Cref{potadd} then \Cref{obs sponsor lemma1}, $\rho^*_H(v_1v_2)+\rho^*_H(v_3v_4)+\rho^*_H(vw_1w_2v_1)\geq \rho^*_H(vv_1v_2v_3v_4w_1w_2) + \rho^*_H(v_1) \geq 3\cdot 4 + 7\cdot 2 - 8 =18$.

Thus, we can suppose w.l.o.g. that $\rho^*_H(v_1v_2)\geq 3$.
\begin{itemize}
\item Suppose that $\rho^*_H(vw_1)\leq 6$ and $\rho^*_H(vw_2)\leq 6$.
\begin{itemize}
\item Suppose that $\rho^*_H(vv_3)\leq 2$ and $\rho^*_H(vv_4)\leq 2$.\\ Then $\rho^*_H(vw_1)+\rho^*_H(vw_2)+\rho^*_H(vv_3)+\rho^*_H(vv_4)+\rho^*_H(v_1) +\rho^*_H(v_2)\leq 2\cdot 6 + 2\cdot 2 + \rho^*_H(v_1)+\rho^*_H(v_2)= 16 + \rho^*_H(v_1)+\rho^*_H(v_2)$. Moreover, by \Cref{potadd}, \Cref{obs sponsor lemma1} then \Cref{claim sponsor lemma1}, $\rho^*_H(vw_1)+\rho^*_H(vw_2)+\rho^*_H(vv_3)+\rho^*_H(vv_4)+\rho^*_H(v_1)+\rho^*_H(v_2)\geq \rho^*_H(vv_1v_2v_3v_4w_1w_2) + 3\rho^*_H(v)\geq 3\cdot 4 + 7\cdot 2 - 8 + 3\cdot 1=21$. As a result, $\rho^*_H(v_1)+\rho^*_H(v_2)\geq 21-16=5$. So, $\rho^*_H(v_1)\geq 3$ w.l.o.g. and by \Cref{potsuperset} $\rho^*_H(v_1v_i)\geq 3$ for all $2\leq i\leq 4$.

If $\rho^*_H(vv_2)\geq 3$ then $\rho^*_H(vv_2v_1v_3)\leq 5$ and $\rho^*_H(vv_2v_1v_4)\leq 5$ by \Cref{claim4 sponsor lemma1}. As a result, $22= 2\cdot 6 + 2\cdot 5 \geq  \rho^*_H(vw_1)+\rho^*_H(vw_2)+\rho^*_H(vv_2v_1v_3)+\rho^*_H(vv_2v_1v_4)\geq \rho^*_H(vv_1v_2v_3v_4w_1w_2) + \rho^*_H(vv_1v_2)+2\rho^*_H(v)\geq 3\cdot 4 + 7\cdot 2 - 8 + 3 + 2\cdot 1=23$ by \Cref{potadd}, \Cref{obs sponsor lemma1} then \Cref{potsuperset} and \Cref{claim sponsor lemma1}.

If $\rho^*(vv_2)\leq 2$, then $18 = 3\cdot 2 + 2\cdot 6\geq \rho^*_H(vv_2)+\rho^*_H(vv_3)+\rho^*_H(vv_4)+\rho^*_H(vw_1)+\rho^*_H(vw_2)\geq \rho^*_H(vv_2v_3v_4w_1w_2)+4\rho^*_H(v)\geq 3\cdot 3 + 7\cdot 2 - 8 + 4\cdot 1 = 19$ by \Cref{potadd}, \Cref{obs sponsor lemma1} then \Cref{claim sponsor lemma1}.

\item Suppose w.l.o.g. that $\rho^*_H(vv_3)\geq 3$.\\
Then, by \Cref{claim4 sponsor lemma1}, $\rho^*_H(vv_3v_1v_2)\leq 5$. As a result, $\rho^*_H(vv_3v_1v_2)+\rho^*_H(vw_1)+\rho^*_H(vw_2)+\rho^*_H(vv_4)\leq 5+2\cdot 6+\rho^*_H(vv_4)=17+\rho^*_H(vv_4)$. Moreover, by \Cref{potadd}, \Cref{obs sponsor lemma1} then \Cref{claim sponsor lemma1}, $\rho^*_H(vv_3v_1v_2)+\rho^*_H(vw_1)+\rho^*_H(vw_2)+\rho^*_H(vv_4)\geq \rho^*_H(vv_1v_2v_3v_4w_1w_2) + 3\rho^*_H(v)\geq 3\cdot 4 + 7\cdot 2 - 8 + 3\cdot 1 = 21$. So, $\rho^*_H(vv_4)\geq 21-17=4$. 

Thus, by \Cref{claim4 sponsor lemma1}, $\rho^*_H(vv_4v_1v_2)\leq 5$. Finally, $22= 2\cdot 6 + 2\cdot 5 \geq  \rho^*_H(vw_1)+\rho^*_H(vw_2)+\rho^*_H(vv_3v_1v_2)+\rho^*_H(vv_4v_1v_2)\geq \rho^*_H(vv_1v_2v_3v_4w_1w_2) + \rho^*_H(vv_1v_2)+2\rho^*_H(v)\geq 3\cdot 4 + 7\cdot 2 - 8 + 3 + 2\cdot 1=23$ by \Cref{potadd}, \Cref{obs sponsor lemma1} then \Cref{potsuperset}, and \Cref{claim sponsor lemma1}. 
\end{itemize}

\item Suppose w.l.o.g. that $\rho^*_H(vw_1)\geq 7$, by \Cref{claim2 sponsor lemma1}, $\rho^*_H(vw_1v_1v_2)\leq 9$.\\
If $\rho^*_H(v_3v_4)\geq 3$, then by \Cref{claim2 sponsor lemma1}, $\rho^*_H(vw_1v_3v_4)\leq 9$. As a result, $\rho^*_H(vw_1v_1v_2)+\rho^*_H(vw_1v_3v_4)+\rho^*_H(vw_2)\leq 2\cdot 9 +\rho^*_H(vw_2)=18 +\rho^*_H(vw_2)$. Moreover, by \Cref{potadd}, \Cref{obs sponsor lemma1} then \Cref{claim sponsor lemma1}, $\rho^*_H(vw_1v_1v_2)+\rho^*_H(vw_1v_3v_4)+\rho^*_H(vw_2)\geq \rho^*_H(vv_1v_2v_3v_4w_1w_2)+ \rho^*_H(vw_1)+\rho^*_H(v)\geq 3\cdot 4 + 7\cdot 2 - 8 + 7+1 = 26$. So, $\rho^*_H(vw_2)\geq 26-18=8$. Thus, by \Cref{claim2 sponsor lemma1}, $\rho^*_H(vw_2v_3v_4)\leq 9$. Finally, $18=2\cdot 9\geq \rho^*_H(vw_1v_1v_2)+\rho^*_H(vw_2v_3v_4)\geq \rho^*_H(vv_1v_2v_3v_4w_1w_2)+\rho^*_H(v)\geq 3\cdot 4 - 7\cdot 2 - 8 + 1=19$ by \Cref{potadd}, \Cref{obs sponsor lemma1}, and \Cref{claim sponsor lemma1}.

If $\rho^*_H(v_3v_4)\leq 2$, then we get $\rho^*_H(vw_1v_1v_2)+\rho^*_H(v_3v_4)+\rho^*_H(vw_2)\leq 9+2+\rho^*_H(vw_2)=11+\rho^*_H(vw_2)$. Moreover, by \Cref{potadd}, \Cref{obs sponsor lemma1} then \Cref{claim sponsor lemma1}, $\rho^*_H(vw_1v_1v_2)+\rho^*_H(v_3v_4)+\rho^*_H(vw_2)\geq \rho^*_H(vv_1v_2v_3v_4w_1w_2)+ \rho^*_H(v)\geq 3\cdot 4 + 7\cdot 2 - 8 + 1 = 19$. As a result, $\rho^*_H(vw_2)\geq 19 -11 = 8$ and by \Cref{claim2 sponsor lemma1}, $\rho^*_H(vw_2v_1v_2)\leq 9$. Finally, $20=2\cdot 9 + 2\geq \rho^*_H(vw_1v_1v_2)+\rho^*_H(vw_2v_1v_2) + \rho^*_H(v_3v_4)\geq \rho^*_H(vv_1v_2v_3v_4w_1w_2)+\rho^*_H(vv_1v_2) \geq 3\cdot 4 + 7\cdot 2 - 8 + 3=21$ by \Cref{potadd}, \Cref{obs sponsor lemma1}, and \Cref{potsuperset}.  
\end{itemize}

\textbf{For $\mathbf{k=5}$:}
\begin{itemize}
\item Suppose that $\rho^*_H(vw_1)\leq 6$.\\
As a result, $\rho^*_H(vw_1) + \rho^*_H(vv_1v_2v_3v_4v_5)\leq 6+\rho^*_H(vv_1v_2v_3v_4v_5)$. Moreover, by \Cref{potadd}, \Cref{obs sponsor lemma1} then \Cref{claim sponsor lemma1}, $\rho^*_H(vw_1) + \rho^*_H(vv_1v_2v_3v_4v_5)\geq \rho^*_H(vv_1v_2v_3v_4v_5w_1) + \rho^*_H(v)\geq 3\cdot 5 + 7 - 8 + 1= 15$. So, $\rho^*_H(vv_1v_2v_3v_4v_5) \geq 15 - 6 = 9$.

By \Cref{potadd}, \Cref{obs sponsor lemma1} then \Cref{claim sponsor lemma1}, $\sum_{i=1}^5 \rho^*(vv_i)\geq \rho^*_H(vv_1v_2v_3v_4v_5) + 4\rho^*_H(v)\geq 9 + 4 = 13$. So, there exists $1\leq i\leq 5$ such that $\rho^*_H(vv_i)\geq 3$. 

Say w.l.o.g. that $\rho^*(vv_1)\geq 3$.

If $\rho^*(v_iv_j)\leq 2$ for all $2\leq i\neq j\leq 5$, then $\rho^*_H(v_2v_3)+\rho^*_H(v_4v_5)+\rho^*_H(vw_1)\leq 2\cdot 2 + 6 = 10$. However, by \Cref{potadd}, \Cref{obs sponsor lemma1} then \Cref{claim sponsor lemma1}, $\rho^*_H(v_2v_3)+\rho^*_H(v_4v_5)+\rho^*_H(vw_1)\geq \rho^*_H(vv_2v_3v_4v_5w_1) \geq 3\cdot 4 + 7 - 8 = 11$. 

So, there exist $2\leq i\neq j\leq 5$ such that $\rho^*_H(v_iv_j)\geq 3$. Say w.l.o.g. that $\rho^*_H(v_2v_3)\geq 3$.

By \Cref{claim4 sponsor lemma1}, $\rho^*_H(vv_1v_2v_3)\leq 5$. Moreover, by \Cref{potadd}, \Cref{obs sponsor lemma1} then \Cref{claim sponsor lemma1}, $\rho^*_H(vv_1v_2v_3)+\rho^*_H(vw_1)+\rho^*_H(v_4v_5) \geq \rho^*_H(vv_1v_2v_3v_4v_5w_1) + \rho^*_H(v) \geq 3\cdot 5 + 7 - 8 + 1 = 15$. As a result, $5+6+\rho^*_H(v_4v_5) \geq \rho^*_H(vv_1v_2v_3)+\rho^*_H(vw_1)+\rho^*_H(v_4v_5) \geq 15$. In other words, $\rho^*_H(v_4v_5)\geq 15-11=4$.

By \Cref{claim4 sponsor lemma1}, $\rho^*_H(vv_1v_4v_5)\leq 5$. As a result, $\rho^*_H(vv_1v_2v_3)+\rho^*_H(vv_1v_4v_5)+\rho^*_H(vw_1)\leq 2\cdot 5 + 6 = 16$. However, by \Cref{potadd}, \Cref{obs sponsor lemma1} then \Cref{claim sponsor lemma1}, $\rho^*_H(vv_1v_2v_3)+\rho^*_H(vv_1v_4v_5)+\rho^*_H(vw_1)\geq \rho^*_H(vv_1v_2v_3v_4v_5w_1) + \rho^*_H(vv_1) + \rho^*_H(v)\geq 3\cdot 5 + 7 - 8 + 3 + 1 = 18$.

\item Suppose that $\rho^*_H(vw_1)\geq 7$.\\
If $\rho^*_H(v_iv_j)\leq 2$ for all $1\leq i\neq j\leq 5$, then $\rho^*_H(v_1v_2)+\rho^*_H(v_3v_4)+\rho^*_H(v_5v_1)\leq 3\cdot 2 = 6$. Moreover, by \Cref{potadd} then \Cref{obs sponsor lemma1} $\rho^*_H(v_1v_2)+\rho^*_H(v_3v_4)+\rho^*_H(v_5v_1)\geq \rho^*_H(v_1v_2v_3v_4v_5) +\rho^*_H(v_1)\geq 3\cdot 5 - 9 + \rho^*_H(v_1)= 6 + \rho^*_H(v_1)$. So, $\rho^*_H(v_1)\leq 6-6 = 0$. Symmetrically, $\rho^*_H(v_i)\leq 0$ for all $2\leq i\leq 5$. However, by \Cref{potadd} then \Cref{obs sponsor lemma1}, $0\geq \sum_{i=1}^5\rho^*_H(v_i)\geq \rho^*_H(v_1v_2v_3v_4v_5) \geq 3\cdot 5 - 9 = 6$. So, w.l.o.g. $\rho^*_H(v_1v_2)\geq 3$.

By \Cref{claim2 sponsor lemma1}, $\rho^*_H(vw_1v_1v_2)\leq 9$. Moreover, by \Cref{potadd}, \Cref{obs sponsor lemma1} then \Cref{claim sponsor lemma1}, $\rho^*_H(vw_1v_1v_2)+\rho^*_H(v_3v_4)+\rho^*_H(v_3v_5) \geq \rho^*_H(vv_1v_2v_3v_4v_5w_1) + \rho^*_H(v_3) \geq 3\cdot 5 + 7 - 8 = 14$. As a result, $9+\rho^*_H(v_3v_4)+\rho^*_H(v_3v_5) \geq \rho^*_H(vw_1v_1v_2)+\rho^*_H(v_3v_4)+\rho^*_H(v_3v_5) \geq 14$. In other words, $\rho^*_H(v_3v_4)+\rho^*_H(v_3v_5)\geq 14-9=5$. So, w.l.o.g. $\rho^*_H(v_3v_4)\geq 3$. 

By \Cref{claim2 sponsor lemma1}, $\rho^*_H(vw_1v_3v_4)\leq 9$. Moreover, by \Cref{potadd}, \Cref{obs sponsor lemma1} then \Cref{claim sponsor lemma1}, $\rho^*_H(vw_1v_1v_2)+\rho^*_H(vw_1v_3v_4)+\rho^*_H(v_3v_5) \geq \rho^*_H(vv_1v_2v_3v_4v_5w_1) + \rho^*_H(vw_1) + \rho^*_H(v_3) \geq 3\cdot 5 + 7 - 8 + 7 = 21$. As a result, $2\cdot 9+\rho^*_H(v_3v_5) \geq \rho^*_H(vw_1v_1v_2)+\rho^*_H(vw_1v_3v_4)+\rho^*_H(v_3v_5) \geq 21$. In other words, $\rho^*_H(v_3v_5)\geq 21-18=3$.

By \Cref{claim2 sponsor lemma1}, $\rho^*_H(vw_1v_3v_5)\leq 9$. As a result, $\rho^*_H(vw_1v_1v_2)+\rho^*_H(vw_1v_3v_4)+\rho^*_H(vw_1v_3v_5)\leq 3\cdot 9= 27$. However, by \Cref{potadd}, \Cref{obs sponsor lemma1} then \Cref{claim sponsor lemma1}, $\rho^*_H(vw_1v_1v_2)+\rho^*_H(vw_1v_3v_4)+\rho^*_H(vw_1v_3v_5)\geq \rho^*_H(vv_1v_2v_3v_4v_5w_1) + 2\rho^*_H(vw_1) + \rho^*_H(v_3)\geq 3\cdot 5 + 7 - 8 + 2\cdot 7 = 28$.
\end{itemize}

\textbf{For $\mathbf{k=6}$:}\\
By \Cref{potadd}, \Cref{obs sponsor lemma1} then \Cref{claim sponsor lemma1}, $\sum_{i=1}^6\rho^*_H(vv_i) \geq \rho^*_H(vv_1v_2v_3v_4v_5v_6) + 5\rho^*_H(v) \geq 3\cdot 6 - 8 +5 =15$. So, there exists $1\leq i\leq 6$ such that $\rho^*_H(vv_i)\geq 3$.

Say w.l.o.g. that $\rho^*_H(vv_1)\geq 3$. 

If $\rho^*_H(v_iv_j)\leq 2$ for all $2\leq i\neq j\leq 6$, then $\rho^*_H(v_2v_3)+\rho^*_H(v_4v_5)+\rho^*_H(v_6v_2)\leq 3\cdot 2 = 6$. Moreover, by \Cref{potadd} then \Cref{obs sponsor lemma1}, $\rho^*_H(v_2v_3)+\rho^*_H(v_4v_5)+\rho^*_H(v_6v_2)\geq \rho^*_H(v_2v_3v_4v_5v_6) +\rho^*_H(v_2)\geq 3\cdot 5 - 9 + \rho^*_H(v_2)$. So, $\rho^*_H(v_2)\leq 6-6 = 0$. Symmetrically, $\rho^*_H(v_i)\leq 0$ for all $2\leq i\leq 6$. However, by \Cref{potadd} then \Cref{obs sponsor lemma1}, $0\geq \sum_{i=2}^6\rho^*_H(v_i)\geq \rho^*_H(v_2v_3v_4v_5v_6) \geq 3\cdot 5 - 9 = 6$. So, there exist $2\leq i\neq j\leq 6$ such that $\rho^*_H(v_iv_j)\geq 3$.

Say w.l.o.g. that $\rho^*_H(v_2v_3)\geq 3$.

By \Cref{claim4 sponsor lemma1}, $\rho^*_H(vv_1v_2v_3)\leq 5$. Moreover, by \Cref{potadd} then \Cref{obs sponsor lemma1}, $\rho^*_H(vv_1v_2v_3)+\rho^*_H(v_4v_5)+\rho^*_H(v_4v_6) \geq \rho^*_H(vv_1v_2v_3v_4v_5v_6) + \rho^*_H(v_4) \geq 3\cdot 6 - 8 = 10$. As a result, $\rho^*_H(v_4v_5)+\rho^*_H(v_4v_6) + 5 \geq \rho^*_H(vv_1v_2v_3)+\rho^*_H(v_4v_5)+\rho^*_H(v_4v_6)\geq 10$. In other words, $\rho^*_H(v_4v_5)+\rho^*_H(v_4v_6)\geq 10-5=5$. So w.l.o.g. $\rho^*_H(v_4v_5)\geq 3$.

By \Cref{claim4 sponsor lemma1}, $\rho^*_H(vv_1v_4v_5)\leq 5$. Moreover, by \Cref{potadd} then \Cref{obs sponsor lemma1}, $\rho^*_H(vv_1v_2v_3)+\rho^*_H(vv_1v_4v_5)+\rho^*_H(v_4v_6) \geq \rho^*_H(vv_1v_2v_3v_4v_5v_6) + \rho^*_H(vv_1) + \rho^*_H(v_4) \geq 3\cdot 6 - 8 + 3 = 13$. As a result, $\rho^*_H(v_4v_6) + 2\cdot 5 \geq \rho^*_H(vv_1v_2v_3)+\rho^*_H(vv_1v_4v_5)+\rho^*_H(v_4v_6)\geq 13$. So, $\rho^*_H(v_4v_6)\geq 13-10=3$.

By \Cref{claim4 sponsor lemma1}, $\rho^*_H(vv_1v_4v_6)\leq 5$. As a result, $\rho^*_H(vv_1v_2v_3)+\rho^*_H(vv_1v_4v_5)+\rho^*_H(vv_1v_4v_6)\leq 3\cdot 5 = 15$. However, by \Cref{potadd} then \Cref{obs sponsor lemma1}, $\rho^*_H(vv_1v_2v_3)+\rho^*_H(vv_1v_4v_5)+\rho^*_H(vv_1v_4v_6) \geq \rho^*_H(vv_1v_2v_3v_4v_5v_6) + 2\rho^*_H(vv_1) + \rho^*_H(v_4) \geq 3\cdot 6 - 8 + 2\cdot 3 = 16$.
\end{itemize}

\end{proof}

\begin{lemma} \label{sponsor lemma2}
Consider $u$ a $7$-vertex that is incident to a unique $3$-path $up_1p_2p_3v$ and let $P=\{p_1,p_2,p_3\}$. Suppose that $\rho^*_{G-P}(u)\leq \rho^*_{G-P}(v)$, that $u$ is adjacent to at least one $2$-path where the other endvertex is a $5^-$-vertex, and that $u$ is adjacent to exactly one vertex $x$ where $d^*_G(x)\leq 12$ (and so $d_G(x)\leq 3$). Then, $u$ has another neighbor that is neither a $(2,2,0)$-vertex nor a $2$-vertex belonging to a $2$-path.
\end{lemma}

\begin{proof}
Suppose by contradiction that $u$ is incident to a unique 3-path $up_1p_2p_3v$, to a unique neighbor $x$ where $d^*_G(x)\leq 12$, to $k$ 2-paths $uq_iq'_iv_i$ for $1\leq i\leq k$ where $k\geq 1$ and $d_G(v_1)\leq 5$, and $l$ $(2,2,0)$-vertices $w_j$ for $1\leq j\leq l$ where each $w_j$ is incident to two 2-paths $w_jr_jr'_jw'_j$ and $w_js_js'_jw''_j$, and finally $k+l=5$. Due to \Cref{3-path lemma}, \Cref{2-path lemma}, \Cref{2-path lemma1}  and the fact that we have no multi-edges, $u$ is distinct from $v,x,t,v_1,\dots,v_k,w_1,\dots,w_l,w'_1,\dots,w'_l,w''_1,\dots,w''_l$, and the endvertices of the 1-path and 2-path incident to $x$ when $x$ is a $(2,1,0)$-vertex. Similarly, for all $1\leq j\leq l$, $w_j$ is distinct from $w'_1,\dots,w'_l,w''_1,\dots,w''_l$.

Let $H=G-(\{u,p_1,p_2,p_3\}\cup \{q_i,q'_i|1\leq i\leq k\})$. 

We claim that:
\begin{claim} \label{claim sponsor lemma2}
For all $1\leq i\leq k$ and $1\leq j\leq l$, $\rho^*_H(vv_i)\leq 2$, $\rho^*_H(vw_j)\leq 6$ and $\rho^*_H(vx)\leq 6$.
\end{claim}

If \Cref{claim sponsor lemma2} holds, then we have the following. Recall that $\rho^*_{G-P}(u)\leq \rho^*_{G-P}(v)$ and by \Cref{pot 3-path lemma}, $\rho^*_{G-P}(v)\geq 1$. As a result, $\rho^*_H(v)\geq \rho^*_{G-P}(v)\geq 1$ by \Cref{potsubgraph}. By \Cref{claim sponsor lemma2}, we have $\sum_{i=1}^k\rho^*_H(vv_i)+\sum_{j=1}^l\rho^*_H(vw_l)+\rho^*_H(vx)\leq 2k+6l+6 = 16 + 4l$. However, by \Cref{potadd}, then by \Cref{potlowerbound} and the fact that $\rho^*_H(v)\geq 1$, we also have $\sum_{i=1}^k\rho^*_H(vv_i)+\sum_{j=1}^l\rho^*_H(vw_l)+\rho^*_H(vx)\geq \rho^*_H(vxv_1\dots v_kw_1\dots w_l) + 5\rho^*_H(v)\geq 14+4l+5 = 19+4l$ which is contradiction.

Now, let us prove \Cref{claim sponsor lemma2}.

First, suppose that $\rho^*_H(vv_{i_0})\geq 3$ for $1\leq i_0\leq k$. We add the 2-path $vp'_1p'_2v_{i_0}$ in $H$, let $P'=\{p'_1,p'_2\}$ and let $H+P'$ be the resulting graph. Since $\rho^*_H(vv_{i_0})\geq 3$, by \Cref{potlemma}, $\mad(H+P')\leq\frac{18}{7}$. By minimality of $G$, there exists a coloring $\psi$ of $H+P'$. We define $\phi$ a coloring of $G$ as follows:
\begin{itemize}
\item If $y\in V(H)\setminus(\{x\}\cup\{w_j,r_j,s_j|1\leq j\leq l\})$, then $\phi(y)=\psi(y)$.
\item Let $\phi(q'_{i_0})=\psi(p'_2)$ unless $i_0=1$ and $\phi(p_3)=\psi(p'_1)$. 
\item We color $q'_i$ for all $i\neq i_0$ and $2\leq i\leq k$ since they all have at least one available color each. 
\item Note that $d^*_G(p_2)=4$, $d^*_G(q'_1)=d_G(v_1)+2\leq 5+2=7$ and $d^*_G(r_j)=d^*_G(s_j)=5$ for all $1\leq j\leq l$ so we can always color them last.
\item Let $L(y)$ be the list of available colors left for a vertex $y$. Observe that we have $|L(u)|\geq 8-(k-1)-3\geq 1$, $|L(x)|\geq 3$ (since $x$ sees $d^*_G(x) - 7\leq 5$ colored vertices), $|L(p_1)|,|L(q_1)|\geq 7$ and $|L(w_1)|,\dots,|L(w_l)|$, $|L(q_2)|,\dots,|L(q_k)|\geq 6$. By \Cref{Hall}, these eight vertices are colorable unless $|L(u)\cup L(x)\cup L(p_1)\cup L(q_1)\cup\dots\cup L(q_k)\cup L(w_1)\cup\dots\cup L(w_l)|=7$. However, $p_1$ sees $\phi(p_3)=\psi(p'_1)\notin\{\psi(p'_2),\psi(v_{i_0})\}=\{\phi(q'_{i_0}),\phi(v_{i_0})\}$ which $q_{i_0}$ sees when $i_0\neq 1$. And when $i_0=1$, $p_1$ sees $\phi(p_3)=\psi(p'_1)\neq \psi(v_1)=\phi(v_1)$ which $q_1$ sees. In both cases, we have $|L(p_1)\cup L(q_{i_0})|\geq 8$.
\end{itemize}

We obtain a valid coloring of $G$ so $\rho^*_H(vv_i)\leq 2$ for all $1\leq i\leq k$.

Now, suppose that $\rho^*_H(vz)\geq 7$ for $z=x$ or $z=w_{j_0}$ for $1\leq j_0\leq l$. We add the edge $e=vz$ in $H$ and let $H+e$ be the resulting graph. Since $\rho^*_H(vz)\geq 7$, by \Cref{potlemma}, $\mad(H+e)\leq \frac{18}7$. By minimality of $G$, there exists a coloring $\psi$ of $H+e$. We define $\phi$ a coloring of $G$ as follows:
\begin{itemize}
\item If $y\in V(H)\setminus(\{x\}\cup\{w_j,r_j,s_j|1\leq j\leq l\})$, then $\phi(y)=\psi(y)$.
\item Let $\phi(p_3)=\phi(z)=\psi(z)$. 
\item We color $q'_i$ for all $2\leq i\leq k$ since they all have at least one available color each. 
\item Note that $d^*_G(p_2)=4$, $d^*_G(q'_1)=d_G(v_1)+2\leq 5+2=7$ and $d^*_G(r_j)=d^*_G(s_j)=5$ for all $1\leq j\leq l$ so we can always color them last.
\item Hence, it remains to color $N_G(u)\cup\{u\}\setminus\{z\}$. Let $L(y)$ be the list of available colors left for a vertex $y$. Observe that we have $|L(u)|\geq 8-(k-1)-1-2\geq 1$, $|L(x)|\geq 2$ if $z=w_{j_0}$ (since $x$ sees $d^*_G(x)-6\leq 6$ colored vertices), $|L(q_1)|\geq 6$ and $|L(w_1)|,\dots,|L(w_l)|,|L(q_2)|,\dots,|L(q_k)|\geq 5$. Since there are six uncolored vertices without counting $p_1$, by \Cref{Hall}, we can color all of these vertices except $p_1$.
\item Finally, we can color $p_1$ since it sees eight colored vertices but it sees $\phi(p_3)=\phi(z)$ twice.
\end{itemize}

We obtain a valid coloring of $G$ so $\rho^*_H(vw_j)\leq 6$ for all $1\leq j\leq l$ and $\rho^*_H(vx)\leq 6$.

\end{proof}

\begin{figure}[H]
\begin{center}
\begin{tikzpicture}{thick}
\begin{scope}[every node/.style={circle,thick,draw,minimum size=1pt,inner sep=2}]
    \node[label={above:$u$}] (0) at (-2,0) {};
    \node[fill,label={above:$q_1$}] (10) at (-1,3.5) {};
    \node[fill,label={above:$q'_1$}] (11) at (0,3.5) {};
    \node[label={above:$v_1$}] (12) at (1,3.5) {$5^-$};
	
	\node[fill,label={above:$q_2$}] (13) at (-1,2.5) {};
    \node[fill,label={above:$q'_2$}] (14) at (0,2.5) {};
    \node[label={above:$v_2$}] (15) at (1,2.5) {};    
    
    \node[fill,label={above:$q_k$}] (20) at (-1,1.5) {};
    \node[fill,label={above:$q'_k$}] (21) at (0,1.5) {};
    \node[label={above:$v_k$}] (22) at (1,1.5) {};
	
	\node[fill,label={above:$w_1$}] (35) at (-1,0) {};    
    \node[fill,label={above:$r_1$}] (30) at (0,0.5) {};
    \node[fill,label={above:$r'_1$}] (31) at (1,0.5) {};
    \node[label={above:$w'_1$}] (32) at (2,0.5) {};
    \node[fill,label={above:$s_1$}] (40) at (0,-0.5) {};
    \node[fill,label={above:$s'_1$}] (41) at (1,-0.5) {};
    \node[label={above:$w''_1$}] (42) at (2,-0.5) {};
    
    \node[fill,label={above:$w_l$}] (55) at (-1,-2) {};
    \node[fill,label={above:$r_l$}] (50) at (0,-1.5) {};
    \node[fill,label={above:$r'_l$}] (51) at (1,-1.5) {};
    \node[label={above:$w'_l$}] (52) at (2,-1.5) {};
    \node[fill,label={above:$s_l$}] (60) at (0,-2.5) {};
    \node[fill,label={above:$s'_l$}] (61) at (1,-2.5) {};
    \node[label={above:$w''_l$}] (62) at (2,-2.5) {};

    \node[fill,label={above:$p_1$}] (3) at (-3,1) {};
    \node[fill,label={above:$p_2$}] (4) at (-4,1) {};
    \node[fill,label={above:$p_3$}] (5) at (-5,1) {};
    \node[label={above:$v$}] (6) at (-6,1) {};
    
    \node[label={above:$x$}] (7) at (-3,-1) {$3^-$};
    
    \node[draw=none] (8) at (-0.5,2.1) {$\dots$};
    
    \node[draw=none] (9) at (-1,-1) {$\dots$};
    
    \node[fill,label={above:$p'_1$}] (100) at (-4,3.8) {};
    \node[fill,label={above:$p'_2$}] (101) at (-2,4.45) {};
\end{scope}

\begin{scope}[every edge/.style={draw=black,thick}]
	\path (0) edge (3);
    \path (0) edge (10);
    \path (0) edge (13);
    \path (0) edge (20);
	\path (0) edge (35);
	\path (0) edge (55);    
    
    \path (35) edge (30);
    \path (35) edge (40);
    \path (55) edge (50);
    \path (55) edge (60);
      \path (10) edge (12);
      \path (13) edge (15);
      \path (20) edge (22);
      \path (30) edge (32);
      \path (40) edge (42);
      \path (50) edge (52);
      \path (60) edge (62);
    
    \path (3) edge (6);
    \path (0) edge (7);
    
    \path (6) edge[bend right=50,dashed] (55);
    \path (6) edge[bend right,dashed] (7);
    \path (6) edge[bend left=60,dashed] (12);
    
\end{scope}
\end{tikzpicture}
\caption{A $7$-vertex that is incident to a $3$-path, $k$ $2$-paths, $l$ $(2,2,0)$-vertices, and a vertex $x$ where $k+l=5$ and $d^*_G(x)\leq 12$.}
\label{sponsor lemma2 figure}
\end{center}
\end{figure}
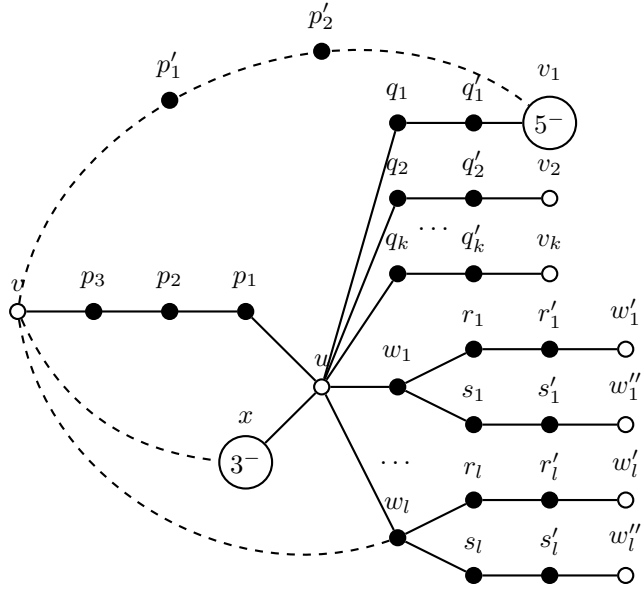

\subsection{Discharging rules \label{tonton}}

In this section, we will define a discharging procedure that contradicts the structural properties of $G$ (see \Cref{counting lemma} to \Cref{sponsor lemma2}) showing that $G$ does not exist. First, we will give a name to some special vertices in $G$.
\medskip
\begin{definition}[Small, medium, and large 2-vertex]
A 2-vertex $v$ is said to be
\begin{itemize}
\item {\em large} if it is adjacent to two $3^+$-vertices,
\item {\em medium} if it is adjacent to exactly one 2-vertex,
\item {\em small} if it is adjacent to two 2-vertices.
\end{itemize}
\end{definition}

\medskip

\begin{definition}[Bridge vertices]
We call a large $2$-vertex, a {\em $1$-path bridge} if it has a $3$-neighbor and a $6^+$-neighbor.
We call two adjacent medium $2$-vertices, a {\em $2$-path bridge} if one has a $5^-$-neighbor and the other a $7$-neighbor.
\end{definition}

\begin{definition}[Sponsor vertex] \label{sponsor definition}
Due to \Cref{tree lemma} and \Cref{three 3-paths lemma}, the $3$-paths in $G$ form a forest of stars. We can thus define the root of each tree in the forest as follows:
\begin{itemize}
\item If a tree is a star with at least two $3$-paths, then the root will be the center of the star.
\item If a tree has only one $3$-path, then let $sp_1p_2p_3r$ be such a $3$-path and $P=\{p_1,p_2,p_3\}$. Suppose w.l.o.g. that $\rho^*_{G-P}(r)\geq \rho^*_{G-P}(s)$. Then, $r$ will be the root (chosen arbitrarily if $\rho^*_{G-P}(r)= \rho^*_{G-P}(s)$).
\end{itemize}
We call a vertex {\em a sponsor} if it is a non-root endvertex of a $3$-path. To each sponsor is assigned the small $2$-vertex on the 3-path connecting it to the root.
\end{definition}

\begin{observation} \label{sponsor observation}
In \Cref{sponsor definition}, the root of a star and the root of a matching are chosen ``differently''. However, if we consider a $3$-path belonging to a star, due to \Cref{pot 3-path lemma} and \Cref{two 3-paths lemma}, the center of the star will always have a higher potential than the sponsor endvertex in the subgraph where we removed the internal $2$-vertices of the $3$-path.
\end{observation}

Now, we assign to each vertex $v$ the charge $\mu(v)=7d(v)-18$. By \Cref{equation}, the total sum of the charges is non-positive. We then apply the following discharging rules:

\medskip

\textbf{Vertices to vertices:}
\begin{itemize}
\item[\ru0] (see \Cref{r0 figure}):
\begin{itemize}
\item[(i)] Every $3^+$-vertex gives 2 to each adjacent large 2-neighbor, and 4 to each adjacent medium 2-neighbor.
\item[(ii)] Every sponsor gives 4 to its small 2-neighbor at distance two.
\item[(iii)] Every $6^+$-vertex gives 1 to each adjacent 1-path bridge.
\item[(iv)] Every 7-vertex gives $\frac12$ to each adjacent 2-path bridge.
\end{itemize}

\item[\ru1] (see \Cref{r1 figure}):
\begin{itemize}
\item[(i)] Every $6^+$-vertex gives 4 to each adjacent (2,2,0)-neighbor.
\item[(ii)] Every $5^+$-vertex gives $\frac52$ to each adjacent (2,1,0)-neighbor.
\item[(iii)] Every $4^+$-vertex gives 1 to each adjacent (1,1,0)-neighbor and $\frac12$ to each adjacent (2,0,0)-neighbor.
\item[(iv)] Every 1-path bridge gives 1 to its 3-neighbor.
\end{itemize}

\item[\ru2] (see \Cref{r2 figure}):
\begin{itemize}
\item[(i)] Every $5^+$-vertex gives $\frac12$ to each (2,2,2,0)-neighbor.
\item[(ii)] Every 2-path bridge gives $\frac12$ to its $5^-$-neighbor.
\end{itemize}

\end{itemize}

\renewcommand{\thesubfigure}{\roman{subfigure}}

\begin{figure}[H]
\begin{center}
\begin{subfigure}[b]{0.3\textwidth}
\begin{center}
\begin{tikzpicture}[baseline=(0.north)]
\begin{scope}[every node/.style={circle,thick,draw,minimum size=1pt,inner sep=2,font=\small}]
    \node[minimum width=23pt] (0) at (1,2) {$3^+$};
    \node[minimum width=23pt] (100) at (-1,2) {$3^+$};
    \node[fill] (3) at (0,2) {};

    \node[minimum width=23pt] (10) at (2,0) {$7$};
    \node[minimum width=23pt] (1100) at (-1,0) {$3^+$};
    \node[fill] (13) at (0,0) {};
    \node[fill] (12) at (1,0) {};

    \node[minimum width=23pt] (30) at (2,1) {$6$};
    \node[minimum width=23pt] (3100) at (-1,1) {$6$};
    \node[fill] (33) at (0,1) {};
    \node[fill] (32) at (1,1) {};

    \node[minimum width=23pt] (20) at (3,-1) {$7$};
    \node[minimum width=23pt] (2100) at (-1,-1) {$7$};
    \node[fill] (23) at (0,-1) {};
    \node[fill] (22) at (2,-1) {};
    \node[fill] (21) at (1,-1) {};
\end{scope}

\begin{scope}[every edge/.style={draw=black,thick}]
    \path (0) edge (100);
      \path[->] (0) edge[bend right] node[above] {2} (3);
      \path[->] (100) edge[bend left] node[above] {2} (3);
        \path (10) edge (1100);
      \path[->] (10) edge[bend right] node[above] {4} (12);
      \path[->] (1100) edge[bend left] node[above] {4} (13);
        \path (30) edge (3100);
      \path[->] (30) edge[bend right] node[above] {4} (32);
      \path[->] (3100) edge[bend left] node[above] {4} (33);
          \path (20) edge (2100);
      \path[->] (20) edge[bend right] node[above] {4} (22);
      \path[->] (2100) edge[bend left] node[above] {4} (23);
\end{scope}
\end{tikzpicture}
\caption{}
\end{center}
\end{subfigure}
\begin{subfigure}[b]{0.3\textwidth}
\begin{center}
\begin{tikzpicture}[baseline=(0.north)]
\begin{scope}[every node/.style={circle,thick,draw,minimum size=1pt,inner sep=2,font=\small}]
    \node[minimum width=23pt,label={above:sponsor}] (0) at (0,0) {$7$};

    \node[minimum width=23pt,label={above:root}] (100) at (-4,0) {$7$};
    \node[fill] (3) at (-3,0) {};
    \node[fill] (2) at (-2,0) {};
    \node[fill] (1) at (-1,0) {};

\end{scope}

\begin{scope}[every edge/.style={draw=black,thick}]
    \path (0) edge (100);
    \path[->] (0) edge[bend right] node [above] {4} (2);
\end{scope}
\end{tikzpicture}
\caption{sponsor case.}
\end{center}
\end{subfigure}
\begin{subfigure}[b]{0.19\textwidth}
\begin{center}
\begin{tikzpicture}[baseline=(0.north)]
\begin{scope}[every node/.style={circle,thick,draw,minimum size=1pt,inner sep=2,font=\small}]
    \node[minimum width=23pt] (0) at (-2,0) {$3$};

    \node[minimum width=23pt] (100) at (-4,0) {$6^+$};
    \node[fill] (3) at (-3,0) {};
\end{scope}

\begin{scope}[every edge/.style={draw=black,thick}]
    \path (0) edge (100);
      \path[->] (100) edge[bend left] node[above] {1} (3);
\end{scope}
\end{tikzpicture}
\caption{1-path bridge case.}
\end{center}
\end{subfigure}
\begin{subfigure}[b]{0.19\textwidth}
\begin{center}
\begin{tikzpicture}[baseline=(0.north)]
\begin{scope}[every node/.style={circle,thick,draw,minimum size=1pt,inner sep=2,font=\small}]
    \node[minimum width=23pt] (10) at (2,0) {$5^-$};
    \node[minimum width=23pt] (1100) at (-1,0) {$7$};
    \node[fill] (13) at (0,0) {};
    \node[fill] (12) at (1,0) {};
    
    \node[ellipse,minimum width=50pt,minimum height=10pt] (14) at (0.5,0) {};
\end{scope}

\begin{scope}[every edge/.style={draw=black,thick}]
        \path (10) edge (1100);
      \path[->] (1100) edge[bend left] node[above] {$\frac12$} (14);
\end{scope}
\end{tikzpicture}
\caption{2-path bridge case.}
\end{center}
\end{subfigure}
\caption{\textbf{R0}.}
\label{r0 figure}
\end{center}
\end{figure}

\begin{figure}[H]
\begin{center}
\begin{subfigure}[b]{0.4\textwidth}
\begin{center}
\begin{tikzpicture}[baseline=(3.north)]
\begin{scope}[every node/.style={circle,thick,draw,minimum size=1pt,inner sep=2,font=\small}]
    \node[minimum width=23pt] (0) at (-4,0) {$6^+$};
    \node[fill] (1) at (-3,0) {};
    
    \node[fill] (2) at (-2,-0.5) {};
    \node[fill] (3) at (-1,-0.5) {};
    \node[minimum width=23pt] (4) at (0,-0.5) {$7$};
    
    \node[fill] (5) at (-2,0.5) {};
    \node[fill] (6) at (-1,0.5) {};
    \node[minimum width=23pt] (7) at (0,0.5) {$7$};
\end{scope}

\begin{scope}[every edge/.style={draw=black,thick}]
    \path (0) edge (1);
	\path (1) edge (2);
	\path (1) edge (5);
	\path (2) edge (4);
	\path (5) edge (7);    
    
    \path[->] (0) edge[bend left] node[above] {4} (1);
\end{scope}
\end{tikzpicture}
\caption{$(2,2,0)$ case.}
\end{center}
\end{subfigure}
\begin{subfigure}[b]{0.4\textwidth}
\begin{center}
\begin{tikzpicture}[baseline=(3.north)]
\begin{scope}[every node/.style={circle,thick,draw,minimum size=1pt,inner sep=2,font=\small}]
    \node[minimum width=23pt] (0) at (-4,0) {$5^+$};
    \node[fill] (1) at (-3,0) {};
    
    \node[fill] (2) at (-2,-0.5) {};
    \node[fill] (3) at (-1,-0.5) {};
    \node[minimum width=23pt] (4) at (0,-0.5) {$7$};
    
    \node[fill] (5) at (-2,0.5) {};
    \node[minimum width=23pt] (7) at (-1,0.5) {$3^+$};
\end{scope}

\begin{scope}[every edge/.style={draw=black,thick}]
    \path (0) edge (1);
	\path (1) edge (2);
	\path (1) edge (5);
	\path (2) edge (4);
	\path (5) edge (7);    
    
    \path[->] (0) edge[bend left] node[above] {$\frac52$} (1);
\end{scope}
\end{tikzpicture}
\caption{$(2,1,0)$ case.}
\end{center}
\end{subfigure}
\begin{subfigure}[b]{0.4\textwidth}
\begin{center}
\begin{tikzpicture}[baseline=(3.north)]
\begin{scope}[every node/.style={circle,thick,draw,minimum size=1pt,inner sep=2,font=\small}]
    \node[minimum width=23pt] (0) at (-4,0) {$4^+$};
    \node[fill] (1) at (-3,0) {};
    
    \node[fill] (2) at (-2,-0.5) {};
    \node[minimum width=23pt] (4) at (-1,-0.5) {$3^+$};
    
    \node[fill] (5) at (-2,0.5) {};
    \node[minimum width=23pt] (7) at (-1,0.5) {$3^+$};

    \node[minimum width=23pt] (10) at (-4,-2) {$4^+$};
    \node[fill] (11) at (-3,-2) {};
    
    \node[fill] (12) at (-2,-2.5) {};
    \node[fill] (13) at (-1,-2.5) {};
    \node[minimum width=23pt] (14) at (0,-2.5) {$7$};
    
    \node[minimum width=23pt] (15) at (-2,-1.5) {$3^+$};
\end{scope}

\begin{scope}[every edge/.style={draw=black,thick}]
    \path (0) edge (1);
	\path (1) edge (2);
	\path (1) edge (5);
	\path (2) edge (4);
	\path (5) edge (7);    
    
    \path[->] (0) edge[bend left] node[above] {1} (1);
    
    \path (10) edge (11);
	\path (11) edge (12);
	\path (11) edge (15);
	\path (12) edge (14);  
    
    \path[->] (10) edge[bend left] node[above] {$\frac12$} (11);
\end{scope}
\end{tikzpicture}
\caption{$(1,1,0)$ and $(2,0,0)$ case.}
\end{center}
\end{subfigure}
\begin{subfigure}[b]{0.4\textwidth}
\begin{center}
\begin{tikzpicture}[baseline=(0.north)]
\begin{scope}[every node/.style={circle,thick,draw,minimum size=1pt,inner sep=2,font=\small}]
    \node[minimum width=23pt] (0) at (-2,0) {3};

    \node[minimum width=23pt] (100) at (-4,0) {$6^+$};
    \node[fill,label={[label distance=5pt]above:bridge}] (3) at (-3,0) {};
\end{scope}

\begin{scope}[every edge/.style={draw=black,thick}]
    \path (0) edge (100);
      \path[->] (3) edge[bend left] node[above] {1} (0);
\end{scope}
\end{tikzpicture}
\caption{1-path bridge case.}
\end{center}
\end{subfigure}

\caption{\textbf{R1}.}
\label{r1 figure}
\end{center}
\end{figure}

\begin{figure}[H]
\begin{center}
\begin{subfigure}[b]{0.49\textwidth}
\begin{center}
\begin{tikzpicture}[baseline=(3.north)]
\begin{scope}[every node/.style={circle,thick,draw,minimum size=1pt,inner sep=2,font=\small}]
    \node[minimum width=23pt] (1) at (-2,0) {$5^+$};
    \node[fill] (2) at (0,0) {};
    
    \node[fill] (3) at (1,0) {};
    \node[fill] (4) at (2,0) {};
    \node[minimum width=23pt] (5) at (3,0) {$7$};
    
    \node[fill] (30) at (1,1) {};
    \node[fill] (40) at (2,1) {};
    \node[minimum width=23pt] (50) at (3,1) {$7$};
    
    \node[fill] (31) at (1,-1) {};
    \node[fill] (41) at (2,-1) {};
    \node[minimum width=23pt] (51) at (3,-1) {$7$};
\end{scope}

\begin{scope}[every edge/.style={draw=black,thick}]    
    \path (1) edge (5);
	
	\path (2) edge (30);
	\path (2) edge (31);
	
	\path (30) edge (50);
	\path (31) edge (51);
		    
    \path[->] (1) edge[bend left] node[above] {$\frac12$} (2);
\end{scope}
\end{tikzpicture}
\caption{$(2,2,2,0)$ case.}
\end{center}
\end{subfigure}
\begin{subfigure}[b]{0.49\textwidth}
\begin{center}
\begin{tikzpicture}[baseline=(0.north)]
\begin{scope}[every node/.style={circle,thick,draw,minimum size=1pt,inner sep=2,font=\small}]
    \node[minimum width=23pt] (10) at (2,0) {$5^-$};
    \node[minimum width=23pt] (1100) at (-1,0) {$7$};
    \node[fill] (13) at (0,0) {};
    \node[fill] (12) at (1,0) {};
    
    \node[ellipse,minimum width=50pt,minimum height=10pt] (14) at (0.5,0) {};
\end{scope}

\begin{scope}[every edge/.style={draw=black,thick}]
        \path (10) edge (1100);
      \path[->] (14) edge[bend left] node[above] {$\frac12$} (10);
\end{scope}
\end{tikzpicture}
\caption{2-path bridge case.}
\end{center}
\end{subfigure}
\caption{\textbf{R2}.}
\label{r2 figure}
\end{center}
\end{figure}

In the following two sections, we will first prove that every vertex ends up with a non-negative charge after the discharging procedure. Thus, by \Cref{equation}, every vertex must have exactly charge 0 which will be proven to be impossible.

\subsection{Verifying that charges on each vertex are non-negative} \label{verification}

Let $\mu^*$ be the assigned charges after the discharging procedure. In what follows, we prove that: $$\forall v \in V(G), \mu^*(v)\ge 0.$$

\begin{enumerate}[\bf C{a}se 1:]
\item $d(v)=2$.\\
We have $\mu(v)=-4$. Vertex $v$ receives 4 by \ru0(i) and \ru(0ii). Now if $v$ belongs to a 1-path (resp. 2-path) bridge, then it also gives 1 (resp. $\frac12$) to a 3-vertex (resp. $5^-$-vertex) by \ru1(iv) (resp. \ru2(ii)), but it also receives 1 (resp. $\frac12$) from \ru0(iii) (resp. \ru0(iv)). In all cases, $\mu^*(v)=0$.

\item $d(v)=3$.\\
Vertex $v$ only gives away charges by \ru0(i): 4 (resp. 2) in the case of a 2-path (resp. in the case of a 1-path) and receives charges by \ru1 and \ru2(ii). Recall $\mu(v)=3$.
By \Cref{small vertex lemma}, $v$ is not a $(2,1^+,1^+)$-vertex. Let us examine all possible configurations for $v$.

\begin{itemize}
\item Suppose that $v$ is a $(2,2,0)$-vertex. Let $v_1$, $v_2$, and $u$ be the two 2-neighbors and $3^+$-neighbor of $v$ respectively. Since $v_1$ and $v_2$ satisfy $d^*(v_1) = d^*(v_2) = 5 \leq \Delta$. By \Cref{counting observation}, $d^*(v)\geq 10$ and $d^*(v)=d(u)+4$, so $d(u)\geq 6$. By \ru1(i), $v$ receives 4 from $u$. Due to \Cref{2-path lemma} and by \ru2(ii), $v$ also receives charge $\frac12$ twice from incident 2-path bridges. In total, we have $$ \mu^*(v) \geq 3-4\cdot 2 + 4 + 2\cdot\frac12 =0.$$

\item Suppose that $v$ is a $(2,1,0)$-vertex. Let $v_1$, $v_2$, and $u$ be the two 2-neighbors (where $v_1$ belongs to the 1-path and $v_2$ belongs to the 2-path) and the $3^+$-neighbor of $v$ respectively. As previously, due to \Cref{2-path lemma} and by \ru2(ii), $v$ receives $\frac12$ from the incident 2-path bridge. Vertex $v_2$ has $d^*(v_2) = 5 \leq \Delta$. By \Cref{counting observation}, $d^*(v)\geq 9$, and $d^*(v)=d(u)+4$, so $d(u)\geq 5$. Hence, $v$ receives $\frac52$ from $u$ by \ru1(ii). So,
$$ \mu^*(v) \geq 3 - 4 - 2 + \frac12 + \frac52 = 0.$$

\item Suppose that $v$ is a $(2,0,0)$-vertex. Let $u_1$, $u_2$, and $v_1$ be the two $3^+$-neighbors and the 2-neighbor of $v$ respectively. Since $d^*(v_1)=5\leq \Delta$. By \Cref{counting observation}, $d^*(v)\geq 9$ and $d^*(v)=d(u_1)+d(u_2)+2$, so $d(u_1)+d(u_2)\geq 7$. We can assume w.l.o.g. that $d(u_1)\geq 4$, thus $v$ receives $\frac12$ from $u_1$ by \ru1(iii). Due to \Cref{2-path lemma} and by \ru2(ii), $v$ also receives $\frac12$ from the incident 2-path bridge. So,
$$ \mu^*(v) \geq 3-4+\frac12 + \frac12=0.$$

\item Suppose that $v$ is a $(1,1,1)$-vertex. Let $vv'v''$ be a 1-path incident to $v$. We have $d^*(v) = 6 \leq \Delta$. It follows that $d^*(v')\geq 9$ by \Cref{counting observation} and as $d^*(v')=d(v'')+3$, we have $d(v'')\geq 6$,  meaning that $v'$ is a 1-path bridge. Thus, vertex $v$ gives 2 to each 2-neighbor by \ru0(i) and receives 1 from each 2-neighbor by \ru1(iv). We have
$$ \mu^*(v) \geq 3-3\cdot 2+3\cdot 1=0.$$

\item Suppose that $v$ is a $(1,1,0)$-vertex. Let $vv_1w_1$ and $vv_2w_2$ be the two 1-paths incident to $v$ and let $u$ be the $3^+$-neighbor of $v$.

If $d(u)=3$, then $d^*(v) = 7 \leq \Delta$. By \Cref{counting observation}, $d^*(v_1)\geq 9$. As $d^*(v_1)=d(w_1)+3$, we have $d(w_1)\geq 6$ meaning that $v$ receives 1 from $v_1$ by \ru1(iv) (and from $v_2$ by symmetry). Hence,
$$\mu^*(v)\geq 3-2\cdot 2 + 2\cdot 1 = 1.$$

If $d(u)\geq 4$, then $v$ receives 1 from $u$ by \ru1(iii).
$$ \mu^*(v)\geq 3-2\cdot2+1=0.$$

\item Suppose that $v$ is a $(1^-,0,0)$-vertex, then at worst, we have
$$ \mu^*(v)\geq 3 - 2 =1.$$

\end{itemize}

\item $d(v)=4$.\\
Vertex $v$ may give 4 (resp. 2, 1, $\frac12$) by \ru0(i) in the case of a 2-path (resp. \ru0(i) in the case of a 1-path, \ru1(iii) in the case of a (1,1,0)-neighbor, \ru1(iii) in the case of a (2,0,0)-neighbor). Recall $\mu(v)=10$. 

By \Cref{small vertex lemma}, $v$ is not a $(2,1^+,1^+,1^+)$-vertex. Hence, $v$ is incident to at most three 2-paths:
\begin{itemize}
\item If $v$ is a $(2,2,2,0)$, then let $v_1,v_2,v_3$ be the 2-neighbors along the three 2-paths and $u$ the last neighbor. Since $d(v_i)=6\leq \Delta$ for all $1\leq i\leq 3$, by \Cref{counting observation}, $d^*(v)\geq 11$. Moreover, $d^*(v)=d(u)+6$ so $d(u)\geq 5$. By \ru2(i), $v$ receives $\frac12$ from $u$. Due to \Cref{2-path lemma} and by \ru2(ii), $v$ also receives $\frac12$ from each incident 2-path bridge. So, $$\mu^*(v)\geq 10 - 3 \cdot 4 + \frac12 + 3\cdot\frac12 = 0.$$

\item If $v$ is a $(2,2,1^-,0)$, then let $v_1,v_2$ be the 2-neighbors along the two 2-paths and $u_1,u_2$ the other two neighbors. Since $d(v_1)=d(v_2)=6\leq \Delta$, by \Cref{counting observation}, $d^*(v)\geq 10$. Moreover, $d^*(v)=d(u_1)+d(u_2)+4$ so $d(u_1)+d(u_2)\geq 6$. Due to \Cref{2-path lemma} and by \ru2(ii), $v$ also receives $\frac12$ from each incident 2-path bridge. 

If $d(u_1)=2$, then $d(u_2)\geq 4$. So, $$\mu^*(v)\geq 10-2\cdot 4 - 2 + 2\cdot \frac12 = 1.$$ 
If $d(u_1)\geq 3$ and $d(u_2)\geq 3$, then at worst, $$\mu^*(v)\geq 10-2\cdot 4 - 2\cdot 1 + 2\cdot\frac12 = 1.$$

\item If $v$ is a $(2,1^-,1^-,0)$, then at worst $$\mu^*(v)\geq 10-4-2\cdot 2 - 1=1.$$

\item If $v$ is a $(1^-,1^-,1^-,1^-)$, then at worst $$\mu^*(v)\geq 10-4\cdot 2 =2.$$
\end{itemize}

\item $d(v)=5$.\\
Vertex $v$ may give 4 (resp. 2, $\frac52$, 1, $\frac12$, $\frac12$) by \ru0(i) in the case of a 2-path (resp. \ru0(i) in the case of a 1-path, \ru1(ii), \ru1(iii) in the case of a (1,1,0)-neighbor, \ru1(iii) in the case of a (2,0,0)-neighbor, \ru2(i)). Recall $\mu(v)=17$.

If $v$ is a $(2,2,2,2,0^+)$, then let $v_1,v_2,v_3,v_4$ be the 2-neighbors along the four 2-paths and $u$ the last neighbor. Since $d(v_i)=7\leq \Delta$ for all $1\leq i\leq 4$, by \Cref{counting observation}, $d^*(v)\geq 12$. Moreover, $d^*(v)=d(u)+8$ so $d(u)\geq 4$. Finally, 
$$\mu^*(v)\ge 17 - 4\cdot 4 -\frac12 =\frac12.$$ 

If $v$ is a $(2,2^-,2^-,1^-,1^-)$, then $v$ may give at most 4, 4, 4, $\frac52$, $\frac52$ along incident edges ; so $$\mu^*(v)\ge 17 - 3\cdot 4 - 2\cdot\frac52 = 0.$$

If $v$ is a $(1^-,1^-,1^-,1^-,1^-)$, then $v$ may give at most $\frac52$ along each incident edge ; so $$\mu^*(v)\ge 17 - 5\cdot\frac52 = \frac92 .$$

\item $d(v)=6$.\\
Observe that $v$ never gives away more than 4 along any edge. Indeed, it may give 4 (resp. 3, 4, $\frac52$, 1, $\frac12$, $\frac12$) by \ru0(i) in the case of a 2-path (resp. \ru0(i) and \ru0(iii) in the case of a 1-path, \ru1(i), \ru1(ii), \ru1(iii) in the case of a (1,1,0)-neighbor, \ru1(iii) in the case of a (2,0,0)-neighbor, \ru2(i)). Recall that $\mu(v)=24$. So, at worst we have 
$$\mu^*(v)\ge 24 - 6\cdot 4 = 0.$$

\item $d(v) = 7$.\\
Observe that every rule except for \ru1(iv) and \ru2(ii) may apply to $v$ and recall that $\mu(v)=31$. Observe that, when $v$ is not a sponsor, the largest amount of charge that $v$ can send away along an edge is $\frac92$ which only happens in the case of a $2$-path bridge by \ru0(i) and \ru0(iv).
\begin{itemize}
\item Suppose that $v$ is not incident to a $3$-path. The case where $v$ is incident to seven $2$-path bridges is impossible due to \Cref{weird cases lemma}. It follows that,
$$ \mu^*(v)\geq 31 - 6\cdot\frac92 - 4 = 0.$$
\item Suppose that $v$ is incident to a $3$-path, then:
\begin{itemize}

\item Suppose that $v$ is not a sponsor. Then, $v$ gives only $4$ to a medium $2$-neighbor on a $3$-path and nothing to the small 2-neighbor at distance 2. Due to \Cref{weird cases lemma}, $v$ cannot be also incident to six $2$-path bridges. So, at worst, we have
$$ \mu^*(v)\geq 31 - 4 - 5\cdot\frac92 - 4 = \frac12.$$

\item Suppose that $v$ is a sponsor. Then, by definition, $v$ is incident to a unique $3$-path. So, $v$ gives 4 to its medium $2$-neighbor and 4 to its small $2$-neighbor at distance 2, which is a total of 8 that $v$ sends away along the 3-path. 

Suppose that $v$ is not incident to any 2-path bridge. Observe that $v$ gives 4 to a neighbor only by \ru0(i) in the case of a 2-path and \ru1(i) in the case of a $(2,2,0)$-neighbor. By \Cref{sponsor observation} and \Cref{sponsor lemma1}, a sponsor has a neighbor different from a $(2,2,0)$-vertex, a 2-vertex on a 2-path and a 2-vertex on a 3-path. In other words, $v$ gives less than 4 to a vertex, which is at most 3 (by \ru0(i) and \ru0(iii) in the case of a 1-path bridge and less than 3 in the other cases). Thus, at worst, we have
$$ \mu^*(v)\geq 31 - 8 - 3 - 5\cdot4 = 0.$$

Suppose that $v$ is incident to a 2-path bridge. Vertex $v$ is incident to at most two 2-path bridges due to \Cref{sponsor observation} and \Cref{sponsor lemma0} ($v$ gives away $\frac92$ along each of these 2-path bridges). Due to \Cref{sponsor observation} and \Cref{sponsor lemma1}, at least one of $v$'s neighbors is not a $2$-vertex belonging to a $2$-path or a $3$-path, nor a $(2,2,0)$-vertex. So $v$ gives less than 4 to at least one neighbor.

If $v$ gives less than 4 to two neighbors, then we have
$$ \mu^*(v)\geq 31 - 8 - 2\cdot\frac92 - 2\cdot3 - 2\cdot4 = 0.$$
So $v$ gives less than 4 to exactly one neighbor. If that amount is at most 2, then at worst, we have 
$$ \mu^*(v)\geq 31 - 8 - 2\cdot\frac92 - 2 - 3\cdot4 = 0.$$
So that amount must be more than 2, so it must be $\frac52$ by \ru1(ii) in the case of a $(2,1,0)$-neighbor, or 3 by \ru0(i) and \ru0(iii) in the case of a 1-path bridge. Both of these cases cannot occur by \Cref{sponsor lemma2}.

\end{itemize}
\end{itemize}

\end{enumerate}

\subsection{Proving the non-existence of $G$.}

In the previous section, we have proven that every vertex has a non-negative amount of charge after the discharging procedure. Since the discharging rules preserve the total amount of charge and the total is non-positive by \Cref{equation}, every vertex must have exactly charge 0. Consequently, we have the following:
\begin{itemize}
\item There are no 3-paths since the endvertex of a 3-path that is not a sponsor always has at least charge $\frac12$ due to \textbf{Case 6} in \Cref{verification}.
\item There are no 7-vertices. Indeed, since there are no 3-paths, a 7-vertex $v$, with final charge 0, must be incident to six $2$-path bridges (where $v$ gives away $\frac92$ along each path) and be adjacent to a $(2,2,0)$-vertex or a $2$-vertex belonging to a $2$-path (where $v$ gives 4 in each case). The former is impossible due to \Cref{weird cases lemma}. In the latter, the endvertex $u$ of the 2-path, which is not a bridge, must be a $7$-vertex by \Cref{weird cases lemma}. Since $u$ also has charge 0, $u$ must also be incident to six 2-path bridges, which is not possible due to \Cref{weird cases lemma3}.

\item There are no 4-vertices or 5-vertices. Indeed, since there are no 7-vertices, due to \Cref{2-path lemma}, a $4$-vertex or $5$-vertex cannot be incident to a 2-path. So by \textbf{Case 3} and \textbf{Case 4} in \Cref{verification}, they always have at least charge 2.
\item There are no 6-vertices. Indeed, by \textbf{Case 5} of \Cref{verification}, a 6-vertex, with final charge 0, must give 4 to each of its neighbors. In other words, its neighbors must be $(2,2,0)$-vertices (which is impossible by \Cref{2-path lemma} and the fact that we have no 7-vertices) or $2$-vertices belonging to $2$-paths (where the other endvertices are $6$-vertices for the same reason). However, by \Cref{weird cases lemma2}, we cannot have a $6$-vertex that is incident to six $2$-paths with $6$-endvertices.

\item There are no 3-vertices. Indeed, if we take a closer look at \textbf{Case 2} of \Cref{verification}, we can observe the following:
\begin{itemize}
\item There are no 3-vertices incident to a $2$-path by \Cref{2-path lemma} and the fact that we have no $7$-vertices.
\item There are no $(1,1,1)$-vertices since the other endvertices of the 1-paths are $6^+$-vertices and we have no $6^+$-vertices.
\item There are no $(1,1,0)$-vertices. Indeed, the $3^+$-neighbor must be a $3$-vertex since there are no $4^+$-vertex, and as a result, the $(1,1,0)$-vertex has at least charge 1 left.
\item There are no $(1^-,0,0)$-vertex since they always have at least charge 1 left.  
\end{itemize} 
\end{itemize}
Finally, $G$ has only $2$-vertices so $G$ must be a cycle which is $2$-distance $8$-colorable. That completes the proof of \Cref{theorem}.

\bibliographystyle{plain}

\end{document}